\documentclass[reqno]{llncs}
\usepackage[utf8]{inputenc}
\usepackage{amsmath}
\usepackage{amsfonts}
\usepackage{amssymb}
\usepackage{enumerate}
\usepackage{subcaption}
\usepackage{algorithm}
\usepackage{algpseudocode}
\usepackage{tikz}
\usepackage{here}
\usepackage{centernot}

\usepackage[left=1.25in, right=1.25in, top=.9in, bottom=.9in]{geometry}
\usepackage[colorlinks=true,breaklinks=true,bookmarks=true,urlcolor=blue,
     citecolor=blue,linkcolor=blue,bookmarksopen=false,draft=false]{hyperref}

\usepackage{cleveref}

\newcommand{\R}{\mathbb R}

\newcommand\veb{{\ve b}}
\newcommand\vecc{{\ve c}}
\newcommand\ved{{\ve d}}

\newcommand\veg{{\ve g}}

\newcommand\ver{{\ve r}}
\newcommand\ves{{\ve s}}

\newcommand\veu{{\ve u}}
\newcommand\vev{{\ve v}}
\newcommand\vew{{\ve w}}
\newcommand\vex{{\ve x}}
\newcommand\vey{{\ve y}}
\newcommand\vez{{\ve z}}

\newcommand{\FourTiTwo}{{\tt 4ti2}}

\newcommand{\eoproof}{\hspace*{\fill} $\square$ \vspace{5pt}}

\newtheorem{innercustomthm}{Definition}
\newenvironment{customthm}[1]
  {\renewcommand\theinnercustomthm{#1}\innercustomthm}
  {\endinnercustomthm}

\DeclareMathOperator{\rank}{rank}

\newcommand{\cross}{\times}

\DeclareMathOperator{\supp}{supp}

\def\ve#1{\mathchoice{\mbox{\boldmath$\displaystyle\bf#1$}}
{\mbox{\boldmath$\textstyle\bf#1$}}
{\mbox{\boldmath$\scriptstyle\bf#1$}}
{\mbox{\boldmath$\scriptscriptstyle\bf#1$}}}

\title{A Polyhedral Model for Enumeration and Optimization over the Set of Circuits}

\author{Steffen Borgwardt\inst{1}\and Charles Viss\inst{2}}

\institute{\email{\href{mailto:steffen.borgwardt@ucdenver.edu}{steffen.borgwardt@ucdenver.edu}};
University of Colorado Denver \and
\email{\href{mailto:charles.viss@ucdenver.edu}{charles.viss@ucdenver.edu}};
University of Colorado Denver 
}

\date{\today}

\begin{document}

\maketitle

\begin{abstract}
Circuits play a fundamental role in polyhedral theory and linear programming. For instance, circuits are used as step directions in various augmentation schemes for solving linear programs or to leave degenerate vertices while running the simplex method. However, there are significant challenges when implementing these approaches: The set of circuits of a polyhedron may be of exponential size and is highly sensitive to the representation of the polyhedron. 

In this paper, we provide a universal framework for enumerating the set of circuits and optimizing over sets of circuits of a polyhedron in any representation---we propose a polyhedral model in which the circuits of the original polyhedron are encoded as extreme rays or vertices. Many methods in the literature and software assume that a polyhedron is in standard form; our framework is a direct generalization. We demonstrate its value by showing that the conversion of a general representation to standard form may introduce exponentially many new circuits.

We then discuss the main advantages of the generalized polyhedral model. It enables the direct enumeration of useful subsets of circuits such as strictly feasible circuits or sign-compatible circuits, as well as optimization over these sets. In particular, this leads to the efficient computation of a steepest-descent circuit, which can be used in an augmentation scheme for solving linear programs or the construction of sign-compatible circuit walks with additional properties.
\end{abstract}


\section{Introduction}

In his seminal paper \cite{g-75}, Jack Graver formalized the concept of \textit{universal test sets} for families of integer and linear programs. These test sets can be used to either verify the optimality of a given solution or find a strictly improving search direction. In the context of integer programming, such a set has become known as the \textit{Graver basis} of an integer matrix.  In this paper, our interest lies in the enumeration and optimization over Graver's universal test set for a linear program: the so-called \textit{circuits} (\Cref{def:circuits}) of the underlying polyhedron. 

Circuits were first introduced by Rockafellar \cite{r-69} as the \textit{elementary vectors} of a subspace, highlighting their role in matroid theory. Namely, the circuits of a linear matroid correspond to the inclusion-minimal sets of linearly dependent columns of a matrix. In the field of mathematical biology, circuits appear as \textit{elementary modes} and are used to describe and analyze pathways in metabolic networks \cite{gk-04,hks-08,k-17,mr-16}. Additionally, circuits are relevant to the study of the \textit{diameters} of polyhedra and the efficiency of the simplex method \cite{bfh-14,bdf-16}. As potential edge directions of a polyhedron, circuits provide a scheme for walking between vertices of a polyhedron in a manner that generalizes edge walks. Although the famous Hirsch conjecture bounding the combinatorial diamter of polyhedra is false in general \cite{s-11}, the related Circuit Diameter Conjecture \cite{bfh-14} remains an open area of research.

Recently, circuits have been used to develop \textit{augmentation schemes} for solving integer and linear programs \cite{dhl-15,hor-13,how-11}. For instance, the so-called \textit{steepest-descent} augmentation scheme of \cite{dhl-15} generalizes the minimum-mean cycle canceling algorithm and can be applied to any bounded linear program \cite{gdl-15}. The \textit{deepest-descent} augmentation scheme of \cite{dhl-15} is guaranteed to take only polynomially many steps \cite{how-11}. However, although several of these algorithms yield promising convergence bounds, a major challenge in implementing them is actually computing the required circuit directions. 

To address this challenge, we are interested in both the enumeration and optimization over sets of circuits. Since a polyhedron may have exponentially many circuits, complete circuit enumeration is hard in general. However, it is open whether or not the problem is solvable in \textit{polynomial total-time} \cite{k-08}, i.e., if the output can be generated in time that is polynomial in both the input and output sizes. On the other hand, methods to directly optimize over  circuits could provide implementations of these circuit augmentation schemes without needing to completely enumerate the set of circuits.

\subsection{Circuit Enumeration}

Several ways to enumerate circuits appear in the literature. For instance, as the universal test set for linear programs proposed by Graver \cite{g-75}, the set of circuits is also known as the \textit{LP Graver test set} of a polyhedron and is a subset of the related IP Graver basis. Thus, methods for enumerating Graver bases may also be used to enumerate circuits \cite{h-03}. Such algorithms are typically variations of Pottier's  geometric completion procedure \cite{p-96}, an $n$-dimensional generalization of Euclid's algorithm. The goal is to find a generating set of vectors for an integer lattice that is minimal with respect to a certain partial order, where the order depends on whether the LP or IP Graver test set is desired~\cite{h-03}. To do this, the algorithms first compute generating sets containing all minimal elements and then remove those elements which are not minimal. Other algorithms enumerate these sets by computing the Gr\"{o}bner bases of certain toric or lattice ideals \cite{st-97}; these approaches can also be translated into Pottier's algorithm \cite{dhk-13}. A significant challenge when implementing these schemes is the sheer size of the initially computed generating sets \cite{o-10}. Although significant speed-up can be achieved using a project-and-lift approach \cite{h-02}, we seek a more direct algorithm for enumerating only the LP Graver test set of a polyhedron.

Since circuits correspond to the elementary modes of a metabolic network, the circuit enumeration problem is also of interest in the field of mathematical biology \cite{gk-04,hks-08,k-17,mr-16}. When a metabolic network is formulated as a polyhedral cone, elementary modes appear as extreme rays. Hence, most current methods for enumerating elementary modes are variations of the Double Description method \cite{gk-04}; implementations are available through the {\tt efmtool} \cite{efmtool} software package.

A significant challenge that arises when enumerating circuits---one that has remained largely unaddressed in the literature---is the fact that the set of circuits is highly sensitive to the representation of the polyhedron. Most current methods for enumerating circuits, such as the state-of-the-art software package \FourTiTwo \ \cite{4ti2}, the LP Graver test set enumeration approaches of \cite{h-03,st-97}, and the Double Description methods for elementary mode enumeration from \cite{gk-04,efmtool}, assume that the input is a polyhedron in standard form: $P = \{ \vex \in \R^n \colon A \vex = \veb, \ \vex \geq \ve0\}$. Additionally, the circuit augmentation schemes of \cite{dhl-15,hor-13,how-11} are described only for polyhedra in standard form. However, as we demonstrate in \Cref{sec:representation}, simply converting a general polyhedron to standard form can introduce exponentially many unwanted directions to its set of circuits, making the already difficult computational task of circuit enumeration even harder. Thus, we are interested in generalized schemes for enumerating the circuits of a polyhedron in any representation.

\subsection{Main Results}

In this paper, we provide a universal framework for the enumeration and optimization over circuits of a general polyhedron $P = \{ \vex \in \R^n \colon A \vex = \veb, B \vex \leq \ved \}$. We begin in \Cref{sec:representation} with a discussion on the relationship between the set of circuits of a polyhedron and its representation. First, we recall the definition of the set of circuits $\mathcal{C}(A,B)$ of $P$ as given in \cite{bfh-14}:
\begin{customthm}{1}
The set of \textit{circuits} $\mathcal{C}(A,B)$ of a polyhedron $P = \{ \vex \in \R^n \colon A \vex = \veb, B \vex \leq \ved \}$ consists of all $\veg \in \ker(A) \setminus \{ \ve 0 \}$ normalized to coprime integer components for which $B \veg$ is support-minimal over $\{B \vex \colon \vex \in \ker(A) \setminus \{\ve0 \} \}$. 
\end{customthm}
\noindent This definition suggests a naive algorithm for enumerating $\mathcal{C}(A,B)$, which we provide in \Cref{alg:naive} as a baseline approach for the general circuit enumeration problem.

Next, we explore the effect that a change in representation has on the set of circuits of $P$. Recall that current methods for enumerating circuits assume that the input polyhedron is in standard form  \cite{4ti2,gk-04,h-03,st-97,efmtool}. We demonstrate in \Cref{alg:standard_form} how to use existing software to enumerate $\mathcal{C}(A,B)$ as a subset of the circuits of a related standard form polyhedron. However, we show that a shortcoming of this approach is that it can require the enumeration and processing of exponentially many unwanted directions.

For a general polyhedron $P = \{ \vex \in \R^n \colon A \vex = \veb, B \vex \leq \ved \}$, a conversion to standard form $P'$ involves both the introduction of slack variables and the split of free variables into their positive and negative parts:
\begin{align*}
P' = \{ (\vex^+, \vex^-, \ves) \in \R^{2n + m_B} \colon A(\vex^+ - \vex^-) = \veb, \ B(\vex^+ - \vex^-) + I \ves = \ved, \  \vex^+, \vex^-, \ves \geq \ve0 \}.
\end{align*} 
In \Cref{thm:representation}, we prove that in a worst-case scenario, performing this conversion increases the number of circuits of the polyhedron by more than a factor of $2^n$. A smarter approach---taken by \Cref{alg:standard_form}---is to only consider a particular submatrix of the constraint matrix from this standard form representation. In doing so, the original circuits of $P$ are preserved while many fewer unwanted circuits are introduced.  Nevertheless, we prove in \Cref{thm:standard_form_conversion} that this still may require the computation and post-processing of exponentially many additional directions. 

To provide an alternative to the naive \Cref{alg:naive} and the indirect \Cref{alg:standard_form}, we propose in \Cref{sec:model} a polyhedral model for representing the set of circuits of the general polyhedron $P$. Specifically, we prove in \Cref{thm:circuits_as_rays} that the circuits of $P$ appear as extreme rays of the following cone $C_{A,B}$:
\begin{align*}
C_{A,B} = \{(\vex, \vey^+, \vey^-) \in \R^{n + 2m_B} \colon A\vex = \ve0, \ B\vex = \vey^+ - \vey^-, \ \vey^+, \vey^- \geq \ve0\}.
\end{align*}
By intersecting $C_{A,B}$ with a normalizing hyperplane, we show in \Cref{thm:circuits_as_vertices} that the circuits of $P$ therefore appear as vertices in the polytope $P_{A,B}$:
\begin{align*}
P_{A,B} = \{(\vex, \vey^+, \vey^-) \in \R^{n + 2m_B} \colon A\vex = \ve0, \ B\vex = \vey^+ - \vey^-, \ ||\vey^+||_1 + ||\vey^-||_1 = 1,   \  \vey^+, \vey^- \geq \ve0\}.
\end{align*}
These results generalize previous standard form models in the literature used to enumerate elementary modes of a metabolic network \cite{gk-04} and to compute augmenting search directions at degenerate vertices in linear programs \cite{gdl-15}. In this generalization, we take into account the more technical details of circuits of general polyhedra compared to standard form.

The polyhedral models $C_{A,B}$ and $P_{A,B}$ may be used to directly enumerate $\mathcal{C}(A,B)$ via any extreme ray or vertex enumeration scheme. We formally outline the procedure in \Cref{alg:polyhedral_model}. This result has important theoretical implications regarding the computational complexity of circuit enumeration: the existence of a polynomial total-time extreme ray or vertex enumeration algorithm would immediately imply the existence a polynomial total-time algorithm for general circuit enumeration. Although vertex enumeration is provably hard for unbounded polyhedra, it is open whether or not a polynomial total-time enumeration scheme exists for bounded polytopes such as $P_{A,B}$ \cite{k-08}.

In \Cref{sec:advantages}, we describe some powerful applications of the polyhedral model $C_{A,B}$. First, we show that certain faces of the model can be used to represent specific subsets of $\mathcal{C}(A,B)$ (\Cref{sec:subset_enumeration}). In particular, given a point $\vex_0 \in P$, we show in \Cref{sec:feasible} that the model may be used to represent only those circuits which are \textit{strictly feasible} at $\vex_0$ with respect to $P$ (\Cref{thm:strictly_feasible}). Similarly, given a direction $\veu \in \ker(A) \setminus \{\ve0\}$, we show in \Cref{sec:sign_compatible} that a face of $C_{A,B}$ may be used to represent only those circuits which are \textit{sign-compatible} with $\veu$ (\Cref{thm:sign-compatible_rays}). Hence, the model may be used to directly enumerate only those circuits of $P$ which exhibit either feasibility or sign-compatibility. This is a vast improvement over a naive approach of enumerating all such circuits in which one uses \Cref{alg:naive} or \Cref{alg:standard_form} to enumerate all of $\mathcal{C}(A,B)$ and then removes those circuits which do not exhibit the desired property via post-processing. 

Another advantage of the polyhedral model is that it enables direct optimization over sets of circuits (\Cref{sec:optimization_over_circuits}). We show in \Cref{sec:steepest_descent} that this allows for the efficient computation of a generalized \textit{steepest-descent circuit} (\Cref{def:steepest}) at any given point in the polyhedron (\Cref{thm:steepest-descent}). Thus, the model may be used to implement a generalization of the steepest-descent augmentation scheme from \cite{dhl-15} for solving linear programs. While the original scheme in \cite{dhl-15} applies only to standard form polyhedra and relies on an oracle for providing  steepest-descent circuits, our generalized approach can be applied to polyhedra in any representation. We note that this approach also generalizes the framework of \cite{gdl-15} for computing augmenting search directions when solving linear programs. 

We prove that our generalized steepest-descent augmentation algorithm exhibits the same desirable behaviors as that of \cite{dhl-15} (\Cref{lem:steepest,lem:non_increasing,lem:sign_change}). In particular, the scheme is guaranteed to take at most $|\mathcal{C}(A,B)|$ augmenting steps (\Cref{cor:steepest_descent_bound}). Using our polyhedral model to compute the required steepest-descent circuits, it follows that the algorithm terminates in strongly polynomial time when applied to a general polyhedron defined by a totally unimodular matrix (\Cref{thm:weakly_polynomial}).

Finally, in \Cref{sec:construct_sign_compatible_sums}, we show how  the polyhedral model can be used in the construction of \textit{sign-compatible circuit walks} between pairs of points in a polyhedron. In the context of combinatorial optimization, such walks correspond to short, intuitive sequences of transitions between solutions of a linear or integer program. A general framework for constructing these walks is outlined in \Cref{alg:construct_sum}. By using our polyhedral model to compute steepest-descent sign-compatible circuits, we can use this algorithm to efficiently construct \textit{$\vecc$-steepest sign-compatible circuit walks} (\Cref{def:f_optimal}) with respect to a given $\vecc \in \R^n$ (\Cref{cor:f_optimal_walks}). In the case where the polyhedron is given by a totally unimodular matrix, we show that all such walks visit only integral points (\Cref{cor:tu_walks}).  

Proof-of-concept implementations of our algorithms for circuit enumeration, steepest-descent~augmentation, and sign-compatible circuit walk construction can be found at \href{url}{https://github.com/charles- viss/circuits}. We also provide small programs to test the enumeration algorithms on randomly generated polyhedra and on dual transportation polyhedra. In such tests, a conversion to standard form (\Cref{alg:standard_form}) is generally not competitive with the naive \Cref{alg:naive} or our polyhedral model approach of \Cref{alg:polyhedral_model}. Furthermore, the naive algorithm is outperformed by our polyhedral model when enumerating only those circuits which are sign-compatible with a random direction. Lastly, the repository contains small toy examples which illustrate how the computation of steepest-descent circuits via our model can be used to solve general linear programs and construct $\vecc$-steepest sign-compatible circuit walks.

\section{Circuits and Polyhedra Representation}\label{sec:representation}

Consider a general polyhedron of the form $P = \{ \vex \in \R^n \colon A \vex = \veb, B \vex \leq \ved \}$ where $A \in \R^{m_A \cross n}$ and $B \in \R^{m_B \cross n}$. We assume $\rank \binom{A}{B} = n$, so that $P$ is pointed. The \textit{set of circuits $\mathcal{C}(A,B)$ of $P$}, as defined in \cite{bfh-14}, is dependent on the constraint matrices $A$ and $B$.
\begin{definition}\label{def:circuits}
The set of \textit{circuits} $\mathcal{C}(A,B)$ of a polyhedron $P = \{ \vex \in \R^n \colon A \vex = \veb, B \vex \leq \ved \}$ consists of all $\veg \in \ker(A) \setminus \{ \ve 0 \}$ normalized to coprime integer components for which $B \veg$ is support-minimal over $\{B \vex \colon \vex \in \ker(A) \setminus \{\ve0 \} \}$. 
\end{definition}
\noindent When $m_A = 0$, it is assumed that $\ker(A) = \R^n$ and we denote the set of circuits $\mathcal{C}_\leq (B)$. For a polyhedron $P = \{ \vex \in \R^n \colon A \vex = \veb, \ \vex \geq 0\}$ in standard form, the set of circuits is denoted $\mathcal{C}(A)$ and simply consists of support-minimal vectors in $\ker(A) \setminus \{\ve0\}$. 

Geometrically, circuits correspond to the directions of one-dimensional subspaces obtained by intersecting any subset of $\dim(P)-1$ facets of $P$ with linearly independent outer normals. It follows that $\mathcal{C}(A,B)$ consists of all potential edge directions of $P$ as the right-hand side vectors $\veb$ and $\ved$ vary \cite{g-75}. Thus, the set of circuits of a polyhedron is symmetric: $\veg$ is a circuit if and only if $-\veg$ is a circuit. Also, note that the normalization used in Definition~\ref{def:circuits} is rather arbitrary---circuits correspond to directions, so for general purposes, any normalization that results in a unique representative modulo positive scalar multiplication can be used. For this reason, we call any positive scalar multiple of a circuit $\veg \in \mathcal{C}(A,B)$ a \textit{circuit direction} of $P$. 

\subsection{A Naive Circuit Enumeration Algorithm}

The fact that circuits correspond to one-dimensional intersections of facets immediately implies the following naive \Cref{alg:naive} for enumerating $\mathcal{C}(A,B)$ which requires only Gaussian elimination. Informally, the algorithm computes the intersections of all subsets of facets of $P$ with size $\dim(P)-1$, obtaining a pair of opposite circuit directions for each intersection with dimension one. The correctness of \Cref{alg:naive} is implied by the following lemma, which can be derived from Proposition 1 in \cite{kps-17}. We include a simple proof for the sake of completeness.

\begin{algorithm}
\caption{Naive Circuit Enumeration}\label{alg:naive}
\begin{algorithmic}[1]
\Procedure{NaiveCircuits}{$A,B$}\Comment{Computes $\mathcal{C}(A,B)$}
\State $S\gets \emptyset$
\For{each $I \subseteq \{1,...,m_B\}$ where $|I| = n - \rank(A) - 1$}
\State $B_I \gets$ the row submatrix of $B$ indexed by $I$
\If{$\rank \binom{A}{B_I} = n-1$}
\State $\veg \gets$ any $\vex \in \ker\binom{A}{B_I} \setminus \{\ve0\}$ normalized to coprime integer components
\If{$\veg \notin S$}
\State $S \gets S \cup \{\veg, -\veg\}$
\EndIf
\EndIf
\EndFor
\State \textbf{return} $S$
\EndProcedure
\end{algorithmic}
\end{algorithm}

\begin{lemma}\label{lem:rank}
Let $P = \{ \vex \in \R^n \colon A \vex = \veb, B \vex \leq \ved \}$ be a pointed polyhedron, let $\veg \in \ker(A) \setminus \{\ve0\}$ be given, and let $B'$ be the maximal row-submatrix of $B$ satisfying $B' \veg = \ve0$. Then $\veg$ is a circuit direction of $P$ if and only if $\rank \binom{A}{B'} = n-1$.  
\end{lemma} 
\begin{proof}
Note first that since $P$ is pointed, $\rank\binom{A}{B'} \leq n-1$ for any $\veg \in \ker(A) \setminus \{\ve0\}$. 

Suppose that $\veg$ is a circuit direction of $P$. If $\rank\binom{A}{B'} < n-1$, there exist rows of $B$ that can be added to $B'$ in order to form a new row-submatrix $B''$ of $B$ such that $\rank \binom{A}{B''} = n-1$. However, there then exists some nonzero $\vey \in \ker \binom{A}{B''}$ which must satisfy $\supp(B\vey) \subsetneq \supp(B\veg)$, contradicting the fact that $\veg$ is a circuit direction of $P$. 

Conversely, if $\rank\binom{A}{B'} = n-1$, it holds that $\ker\binom{A}{B'}$ is one-dimensional and generated by $\veg$. Thus, any $\vey \in \ker(A) \setminus \{\ve0\}$ satisfying $\supp(B\vey) \subseteq \supp(B\veg)$ must be a scalar multiple of $\veg$, implying that $\veg$ is a circuit direction of $P$. 
\eoproof  \end{proof}

Significant speed-up of \Cref{alg:naive} can be achieved by performing the necessary row-reduction operations in a strategic manner such as a variant of the Double Description method \cite{f-98} or the binary approach proposed in \cite{gk-04}. Of course, the best-case running time for \Cref{alg:naive} is exponential, but it serves as a baseline approach for the general circuit enumeration problem.


\subsection{Enumeration via Standard Form Representation}

An often neglected issue in circuit enumeration is the fact that the set of circuits is sensitive to the representation of a polyhedron. Current state-of-the-art software packages used to enumerate circuits \cite{4ti2} and elementary modes \cite{efmtool}, as well as the other discussions on circuit enumeration in the literature \cite{gk-04,h-03,st-97} are restricted to standard form representation: $P = \{ \vex \in \R^n \colon A \vex = \veb, \ \vex \geq 0\}$. In this situation, the procedure of \Cref{alg:naive} corresponds to finding the inclusion-minimal sets of linearly dependent columns of $A$. 

It is indeed possible to enumerate the set of circuits of a general polyhedron as a subset of the circuits of a related standard form polyhedron---simply test whether or not each standard form circuit satisfies the condition of \Cref{lem:rank} with respect to the original representation. This yields \Cref{alg:standard_form} for general circuit enumeration which can take advantage of existing software packages for computing circuits. The correctness of the algorithm follows from \Cref{lem:rank} and \Cref{thm:standard_form_conversion}, given at the end of the section. After enumerating an initial set of standard form circuits, the algorithm includes a post-processing step in which a rank computation is performed for each element of $\mathcal{C}\binom{A \ \ve0}{B \ I}$.

\begin{algorithm}
\caption{Circuit Enumeration via Standard Form Conversion}\label{alg:standard_form}
\begin{algorithmic}[1]
\Procedure{StandardFormCircuits}{$A,B$}\Comment{Computes $\mathcal{C}(A,B)$}
\State $S \gets \emptyset$
\State Use any standard form circuit enumeration algorithm to compute $\mathcal{C}\binom{A \ \ve0}{B \ I}$
\For{each $(\veg, B\veg) \in \mathcal{C}\binom{A \ \ve0}{B \ I}$}
\State $B' \gets$ the maximal row-submatrix $B'$ of $B$ such that $B' \veg = \ve0$
\If{$\rank \binom{A}{B'} = n-1$}
\State $S \gets S \cup \{\veg\}$
\EndIf
\EndFor
\State \textbf{return} $S$
\EndProcedure
\end{algorithmic}
\end{algorithm}

However, by introducing new facets to the polyhedron, the conversion to standard form required by \Cref{alg:standard_form} increases the number of circuits computed from $|\mathcal{C}(A,B)|$ to $\left|  \mathcal{C}\binom{A \ \ve0}{B \ I} \right|$, making the already difficult task of circuit enumeration even harder. To illustrate this phenomenon, consider a full-dimensional polyhedron $P = \{ \vex \in \R^n \colon B \vex \leq \ved\}$. Unless the constraint matrix $B$ has many subdeterminants equal to zero, converting $P$ to standard form by adding slack variables and splitting up original variables into positive and negative parts will introduce exponentially many unwanted circuits.

\begin{theorem}\label{thm:representation}
Let $P = \{ \vex \in \R^n \colon B \vex \leq \ved\}$ be a polyhedron in which all subdeterminants of $B \in \R^{m_B \cross n}$ are nonzero, and let 
\begin{align*}
P' = \{ (\vex^+, \vex^-, \ves) \in \R^{2n + m_B} \colon B(\vex^+ - \vex^-) + I \ves = \ved, \ \vex^+, \vex^-, \ves \geq \ve0 \}
\end{align*}
be the standard form representation of $P$. Then
\begin{align*}
|\mathcal{C}(P')| = 2n + 2\sum_{d=1}^n \binom{n}{d} \binom{m_B}{d-1} 2^d = 2^n |\mathcal{C}_\leq(B)|  + 2n + 2\sum_{d=1}^{n-1} \binom{n}{d} \binom{m_B}{d-1} 2^d,
\end{align*}
where $\mathcal{C}(P') := \mathcal{C}(B  \ -B  \  \ I)$ is the set of circuits of $P'$. 
\end{theorem}

\begin{proof}
Lemma~\ref{lem:rank} implies that a circuit direction of $P$ corresponds to the one-dimensional kernel of an $(n-1) \cross n$ submatrix of $B$. Since all subdeterminants of $B$ are nonzero, each such submatrix has full row rank and thus has a one-dimensional kernel generated by a circuit of $P$. Furthermore, all such kernels are distinct, else $B$ would contain an $n \cross n$ singular submatrix. Therefore, each set of $n-1$ rows of $B$ yields a pair of opposite circuit directions, implying
\begin{align*}
|\mathcal{C}_\leq (B)| = 2 \binom{m_B}{n-1}.
\end{align*}

Now consider the standard form representation $P'$. The set of circuits of $P'$   consists of those $(\veg^+, \veg^-, \ves) \in \R^{2n + m_B}$ (normalized to coprime integer components) for which $B(\veg^+ - \veg^-)  = -\ves$ and $(\veg^+, \veg^-, \ves)$ is support-minimal over all such vectors. If $(\veg^+, \veg^-, \ves)$ is a circuit of $P'$ and there exists an index $i$ such that $\veg^+_i, \veg^-_i$ are both nonzero, the support of $(\veg^+, \veg^-, \ves)$  can be reduced while maintaining $B(\veg^+ - \veg^-)  = -\ves$ by shifting $\veg^+_i, \veg^-_i$ in the appropriate directions (decreasing one while increasing the other by the same value) so that at least one is equal to zero. Therefore, any circuit with both $\veg^+_i, \veg^-_i \neq 0$ must have $\veg^+_i = \veg^-_i = \pm 1$ with all other components equal to zero. This yields $2n$ circuits of $P'$. If we set $\veg := \veg^+ - \veg^-$, all other circuits of $P'$ then correspond to vectors $\veg \in \R^n$ for which $(\veg, B\veg)$ is support-minimal. (Note the contrast with the original circuits $\mathcal{C}_\leq(B)$---the support of $\veg$ is now taken into account along with the support of $B \veg$.)

Clearly the unit vectors $\mathbf{e}_i$ correspond to circuits of $P'$ since no $\veg \in \R^n \setminus \{\ve0\}$ has strictly smaller support and each $B \mathbf{e}_i$ has only nonzero entries. Each $\mathbf{e}_i$ can be represented in two ways as the difference $\veg^+ - \veg^-$ by selecting one of $\veg^+_i$ or $\veg^-_i$ to be nonzero. Counting both members of opposite pairs, this yields another $2n \cdot 2$ circuits of $P'$.  

Now let $B'$ be a $(d-1) \cross d$ submatrix of $B$ for $d \in {2,...,n}$. Since $B'$ has only nonzero subdeterminants, it has full row rank and its kernel is generated by some $\veg' \in \R^d$. Additionally, each component of $\veg'$ must be nonzero, else a proper subset of the columns of $B'$ would form a singular matrix. Now consider the natural augmentation $\veg \in \R^n$ of $\veg'$, where the components of $\veg$ that do not correspond to columns of $B'$ are equal to zero and the remaining components are equal to those of $\veg'$. Then $(\veg, B\veg)$ corresponds to a circuit of $P'$. To see this, note that any $\vey' \in \R^d \setminus \{\ve0\}$ with $\supp(\vey') \subsetneq \supp(\veg')$ must satisfy $B' \vey' \neq \ve0$. Hence, $\supp((\vey, B\vey)) \centernot \subseteq \supp((\veg, B\veg))$ for the natural augmentation $\vey \in \R^n$ of $\vey'$.

We now show that each such $\veg$ found in the previous paragraph is distinct. If $\veg'$ generates the kernel of a $(d-1) \cross d$ submatrix $B'$ of $B$, then $B'_i \veg' \neq 0$ for any other sub-row $B'_i$ of $B$ corresponding to the columns of $B'$. Otherwise, we could add this row to $B'$ to form a singular $d \cross d$ submatrix of $B$. Hence, if $\veg \in \R^n$ is the natural augmentation of $\veg'$, the only components of $B\veg$ that are zero are those corresponding to rows of $B'$. This implies that each $(d-1) \cross d$ submatrix of $B$ yields a distinct direction $(\veg, B\veg)$. Furthermore, as $\veg$ has exactly $d$ nonzero entries, each resulting $\veg$ can be represented in $2^d$ ways as the difference $\veg^+ - \veg^-$. Therefore, since all possible supports of $(\veg, B\veg)$ have been accounted for, we have:
\begin{align*}
|\mathcal{C}(P')| &= 2n + 2n \cdot 2 + 2\sum_{d=2}^n \binom{n}{d} \binom{m_B}{d-1} 2^d \\
 &= 2n + 2\sum_{d=1}^n \binom{n}{d} \binom{m_B}{d-1}2^d \\
&= 2^n|\mathcal{C}_\leq(B)|  + 2n + 2\sum_{d=1}^{n-1} \binom{n}{d} \binom{m_B}{d-1}2^d.
\end{align*} 
\eoproof  \end{proof}

Note that almost all real matrices $B \in \R^{m_B \cross n}$ satisfy the subdeterminant condition of Theorem~\ref{thm:representation}---while $B$ has a subdeterminant equal to 0, perturb an entry in $B$ to change the determinant of a smallest singular submatrix of $B$ while maintaining the nonzero status of other subdeterminants. However, most integer matrices in combinatorial optimization have many subdeterminants equal to zero, so the number of circuits introduced via standard form conversion will not necessarily be the amount stated in Theorem~\ref{thm:representation}. Nevertheless, the potential to compute a set whose size is more than $2^n$ times the number of actual circuits implies that this conversion should be avoided.

An alternative is to compute the much smaller set of circuits $\mathcal{C}( B \  \ I) $ instead of $\mathcal{C}(B  \ -B  \  \ I)$. This is the approach taken in \Cref{alg:standard_form}. Namely, in the case of a full-dimensional polyhedron $P = \{ \vex \in \R^n \colon B \vex \leq \ved\}$, we avoid splitting variables into their positive and negative parts by considering the standard form polyhedron 
\begin{align*}
P'' = \{ (\vex, \ves) \in \R^{n + m_B} \colon B \vex + I \ves =  \ved, \  \vex, \ves \geq \ve0\}.
\end{align*}
Although $P''$ may no longer be equivalent to the original polyhedron $P$, the circuits of $P$ still appear as a subset of the circuits of $P''$, which now correspond to support-minimal vectors of the form $(\veg, B\veg)$. The exponential term $2^d$ in Theorem~\ref{thm:representation} associated with the splitting of $\veg$ into $\veg^+ - \veg^-$ is avoided, and we obtain the following bound on the number of introduced circuits.
 
\begin{corollary}
Let $P = \{ \vex \in \R^n \colon B \vex \leq \ved\}$ be a polyhedron in which all subdeterminants of $B \in \R^{m_B \cross n}$ are nonzero. Then
\begin{align*}
|\mathcal{C}(B \ \ I)| = 2\sum_{d=1}^n \binom{n}{d} \binom{m_B}{d-1}  =  |\mathcal{C}_\leq(B)| + 2\sum_{d=1}^{n-1} \binom{n}{d} \binom{m_B}{d-1}.
\end{align*}
\end{corollary}

However, we still obtain exponentially many circuits that do not correspond to circuits of $\mathcal{C}_\leq(B)$. In the case where $m_B$ is close to $n$ (for example $m_B = n+1$), the number of introduced circuits is again an exponential multiple of the original number of circuits. For a general polyhedron, we obtain the following bound on the number of circuits introduced using this approach of \Cref{alg:standard_form}.

\begin{theorem}\label{thm:standard_form_conversion}
Let $P = \{ \vex \in \R^n \colon A \vex = \veb, B \vex \leq \ved \}$ be a pointed polyhedron and let $r := \rank(A)$. Then $\veg \in \mathcal{C}(A,B)$ only if $(\veg, B\veg) \in \mathcal{C}\binom{A \ \ve0}{B \ I}$. Furthermore,
\begin{align*}
|\mathcal{C}(A,B)| \leq \left| \mathcal{C}\binom{A \ \ve0}{B \ I} \right| \leq |\mathcal{C}(A,B)| + 2\sum_{d=r+1}^{n-1} \binom{n}{d} \binom{m_B}{d-r-1}, 
\end{align*}
where either bound may be sharp.
\end{theorem} 

\begin{proof}
The elements of $\mathcal{C}(A,B)$ are those of $\veg \in \ker(A) \setminus \{\ve0\}$ such that $B\veg$ is support-minimal over $\{ B\veg \colon \veg \in \ker(A) \setminus \{\ve0\} \}$, while $\mathcal{C}(A,B)$ consists of $\veg \in \ker(A) \setminus \{\ve0\}$ such that $(\veg, B \veg)$ is support-minimal over $\{ (\veg, B \veg) \colon \veg \in \ker(A) \setminus \{ \ve0\} \}$. Clearly if $\veg \in \mathcal{C}(A,B)$ then $(\veg, B\veg) \in \mathcal{C}\binom{A \ \ve0}{B \ I}$, which implies the stated lower bound of the theorem. In fact, if $B = I$ or if $B$ contains $I$ as a row-submatrix, then $\veg \in \mathcal{C}(A,B)$ if and only if $(\veg, B\veg) \in \mathcal{C}\binom{A \ \ve0}{B \ I}$, which implies that the lower bound may be sharp.

On the other hand, the maximum size of $\mathcal{C}\binom{A \ \ve0}{B \ I}$ is 
\begin{align*}
2\sum_{d=r+1}^{n} \binom{n}{d} \binom{m_B}{d-r-1},
\end{align*}
which is achieved when all subdeterminants of $\binom{A}{B}$ are nonzero. To see this, choose $d \in \{r+1,...,n \}$, choose a set of $d$ columns of $A$ to form a column-submatrix $A'$, and choose $d-r-1$ rows of $B$. Form the matrix $\binom{A'}{B'}$, where $B'$ consists of the chosen $d-1$ rows of $B$ restricted to the columns $A'$. Then $\rank\binom{A'}{B'} = d-1$ and $\ker\binom{A'}{B'}$ is generated by some $\veg' \in \R^d$. As in the proof of \Cref{thm:representation}, extend $\veg'$ to the corresponding $\veg \in \R^n$ so that $(\veg, B\veg)$ is a circuit direction of $\mathcal{C}\binom{A \ \ve0}{B \ I}$. This follows from the fact that all subdeterminants are nonzero, which also implies that each such $\veg$ is unique. This accounts for all possible supports of $(\veg, B\veg)$, yielding the stated value for $\left | \mathcal{C}\binom{A \ \ve0}{B \ I} \right|$.

The circuits $(\veg, B\veg)$ formed by choosing $d = n$ in the procedure of the previous paragraph satisfy $\veg \in \mathcal{C}(A,B)$ by \Cref{lem:rank}. Any other circuit of $\mathcal{C}\binom{A \ \ve0}{B \ I}$ must result from a choice of $d \in \{r+1,...,n-1 \}$, and we obtain the stated upper bound for general $\binom{A}{B}$.
\eoproof  \end{proof}

Hence, as described in \Cref{alg:standard_form}, any software that enumerates $\mathcal{C}\binom{A \ \ve0}{B \ I}$, the set of circuits of a standard form polyhedron, may be used to enumerate $\mathcal{C}(A,B)$. However, this can require the enumeration and post-processing of exponentially many unwanted directions.

To complete the discussion on the relationship between the representation of a polyhedron and its set of circuits, consider the conversion of a general polyhedron $P = \{ \vex \in \R^n \colon A \vex = \veb, B \vex \leq \ved \}$ to its canonical form: $P' = \{ \vex \in \R^n \colon A \vex \leq \veb, -A\vex \leq -\veb, B \vex \leq \ved \}$. The set of circuits of $P'$ is 
\begin{align*}
\mathcal{C}_\leq\left( \begin{array}{c}
A \\ -A \\ B
\end{array} \right) = \mathcal{C}_\leq \left( \begin{array}{c}
A  \\ B
\end{array} \right),
\end{align*}
which contains directions that leave $\ker(A)$. In the case where $\binom{A}{B}$ has only nonzero subdeterminants, it holds by \Cref{lem:rank} that any subset of $n-1$ rows of $\binom{A}{B}$ yields a circuit of $\mathcal{C}_\leq \binom{A}{B}$ instead of only those subsets containing all of $A$. If we let $d$ denote the number of rows from $A$ to be included in such a subset, it then follows that 
\begin{align*}
    \left| \mathcal{C}_\leq \binom{A}{B} \right| = 2 \binom{m_A + m_B}{n-1}  &= 2 \sum_{d=0}^{m_A} \binom{m_A}{d}\binom{m_B}{n-1-d} \\ 
    &= \left| \mathcal{C}(A, B) \right| + 2 \sum_{d=0}^{m_A - 1} \binom{m_A}{d}\binom{m_B}{n-1-d}.
\end{align*}
Hence, converting to canonical form may again introduce exponentially many unwanted circuits.

\section{A Polyhedral Model for the Set of Circuits}\label{sec:model}

In this section, we propose a polyhedral model for the set of circuits of a general polyhedron. It may be used to directly enumerate or optimize over either the entire set of circuits or certain useful subsets (see \Cref{sec:advantages}). As a generalization of the results in \cite{gk-04}, where elementary modes of a metabolic network are computed as extreme rays of a polyhedral cone, the support-minimality property of \Cref{def:circuits} yields a method to model the circuits of any  polyhedron---regardless of representation---as extreme rays or vertices of a related polyhedron. First, for a general polyhedron $P = \{ \vex \in \R^n \colon A \vex = \veb, B \vex \leq \ved \}$, we show that the circuits of $P$ appear as extreme rays of the related cone $C_{A,B}$.

\begin{theorem}\label{thm:circuits_as_rays}
Let $P = \{ \vex \in \R^n \colon A \vex = \veb, B \vex \leq \ved \}$ be a pointed polyhedron. The pointed cone
\begin{align*}
C_{A,B} = \{(\vex, \vey^+, \vey^-) \in \R^{n + 2m_B} \colon A\vex = \ve0, \ B\vex = \vey^+ - \vey^-, \ \vey^+, \vey^- \geq \ve0\}.
\end{align*}
is generated by the set of extreme rays $S \cup T'$, where:
\begin{enumerate}
\item The set $S := \{ (\veg, \vey^+, \vey^-) \colon \veg \in \mathcal{C}(A, B), \ \vey_i^+ = \max\{(B\veg)_i,0\}, \  \vey_i^- = \max\{-(B\veg)_i,0\}\}$ gives the circuits of $P$.
\item The set $T'$ is a subset of  $T := \{ (\ve0, \vey^+, \vey^-) \colon \vey^+_i = \vey^-_i = 1 \text{ for some $i \leq m_B$}, \  \vey^+_j = \vey^-_j = 0 \text{ for $j \neq i$}\}$ and has size at most $m_B$.
\end{enumerate}
\end{theorem}

\begin{proof}
Since $P$ is pointed, it holds that $\rank\binom{A}{B} = n$ and  that $C_{A,B}$ is a pointed cone. In particular, there does not exist $(\vex, \vey^+, \vey^-) \in C_{A,B}$ such that $\vex \neq \ve0$ and $\vey^+ - \vey^- = \ve0$. We characterize the extreme rays of $C_{A,B}$. To do this, consider a canonical representation: $C_{A,B} = \{ \ver \in \R^{n + 2m_B} \colon M \ver \geq \ve0 \}$, where
\[
M =
		\left(\begin{array}{rrr}
			A & \ve0 & \ve0 \\ -A & \ve0 & \ve0 \\ B & -I & I \\ -B & I & -I \\ \ve0 & I & \ve0 \\ \ve0 & \ve0 & I
		\end{array} \right)
	\quad \text{and} \quad
	\ver=
		\left(\begin{array}{r}
			\vex^{ \ } \\ \vey^+ \\ \vey^-
		\end{array} \right).
\]
For any $\ver \in C_{A,B}$, let $Z(\ver)$ denote the \textit{zero set} for $\ver$: the set of indices for which the corresponding inequalities in the system $M\ver \geq \ve0$ are satisfied with equality. A useful characterization of the extreme rays of a pointed cone, given by \cite{f-98} and used in \cite{gk-04}, implies that $\ver$ is an extreme ray of $C_{A,B}$ if and only if any nonzero $\ver' \in C_{A,B}$ with $Z(\ver') \supseteq Z(\ver)$ satisfies $\ver' = \alpha \ver$ for some $\alpha > 0$, and thus $Z(\ver') = Z(\ver)$. For any $\ver \in C_{A,B}$, the only inequalities of $M\ver \geq \ve0$ that may not be active are those corresponding to the constraints $\vey^+, \vey^- \geq \ve0$. Therefore, a nonzero $\ver := (\vex, \vey^+, \vey^-) \in C_{A,B}$ is an extreme ray of $C_{A,B}$ if and only if any nonzero $\ver' := (\vex', \vey'^+, \vey'^-) \in C_{A,B}$ with $\supp((\vey'^+, \vey'^-)) \subseteq \supp((\vey^+, \vey^-))$ satisfies $\ver' = \alpha \ver$ for some $\alpha > 0$. We show that the vectors of $S$ satisfy this property and that all such rays belong to $S \cup T$. 

\vspace{.1in}

First, let $\ver := (\veg, \vey^+, \vey^-) \in S$ be given. Since $\veg \in \ker(A)$ and $B\veg = \vey^+ - \vey^-$, we have $\ver \in C_{A,B}$. Let some nonzero $\ver' := (\vex', \vey'^+, \vey'^-) \in C_{A,B}$ such that $\supp((\vey'^+, \vey'^-)) \subseteq \supp((\vey^+, \vey^-))$ be given. Then $\supp(B\vex') \subseteq \supp(B\veg)$. Note that $\vex'$ must be nonzero since $\ver' \neq \ve0$. Therefore, if $\supp(B\vex') \subsetneq \supp (B\veg)$, the fact that $\veg$ is a circuit is contradicted. Hence, we have $\supp(B\vex') = \supp(B\veg)$.

It then must hold that $\vex' = \alpha \veg$ for some $\alpha \in \R$ \cite{g-75}. To see this, consider an index $i$ such that $(B\veg)_i \neq 0$ and hence $(B \vex')_i \neq 0$, and let $\vez := \veg - \frac{(B\veg)_i}{(B\vex')_i} \vex'$. It follows that $(B \vez)_i = 0$, and by the support-minimality of $B \veg$, this implies $\vez = \ve0$. Thus, $\vex' = \alpha \veg$ with $\alpha := \frac{(B\vex')_i}{(B\veg)_i}$. 

Since $\supp((\vey'^+, \vey'^-)) \subseteq \supp((\vey^+, \vey^-))$, the vectors $B\vex'$ and $B\veg$ belong to the same orthant of $\R^{m_B}$. Hence, $\alpha > 0$ and $\ver' = \alpha \ver$. Therefore, any $\ver \in S$ is an extreme ray of $C_{A,B}$.

\vspace{.1in}

On the other hand, it need not hold that all vectors of $T$ are extreme rays of $C_{A,B}$. To see this, let $\ver := (\ve0, \vey^+, \vey^-) \in T$ be given and let $i$ be the index such that $\vey^+_i = \vey^-_i = 1$. Then $\supp((\vey^+, \vey^-))$ consists of two elements. It is possible for there to exist a circuit $\veg \in \ker(A)$ such that $\supp(B \veg)$ consists of only the single index $i$. In this case, there exists a vector $\ver' := (\veg, \vey'^+, \vey'^-) \in S$ where the support of $(\vey'^+, \vey'^-)$ consists  of only one of the two elements of $\supp((\vey^+, \vey^-))$. Hence, $\ver'$ is an extreme ray of $C_{A,B}$ while $\ver$ is not.

\vspace{.1in}

For the reverse direction, we show that any extreme ray of $C_{A,B}$ has a positive scalar multiple in $S \cup T$. Let $\ver := (\vex, \vey^+, \vey^-)$ be an extreme ray of $C_{A,B}$: a nonzero vector in $C_{A,B}$ such that any $\ver':= (\vex', \vey'^+, \vey'^-)  \in C_{A,B} \setminus \{\ve0\}$ with $\supp((\vey'^+, \vey'^-)) \subseteq \supp((\vey^+, \vey^-))$ satisfies $\ver' = \alpha \ver$ for some $\alpha > 0$, implying $\supp((\vey'^+, \vey'^-)) = \supp((\vey^+, \vey^-))$. Note again that both $\vey^+$ and $\vey^-$ cannot be zero since $C_{A,B}$ is pointed.

Suppose there exists an index $i \leq m_B$ such that both $\vey^+_i$ and $\vey^-_i$ are positive. Assume first that $\vey^+_i > \vey^-_i$. Consider the translated vector $\ver' := (\vex, \vey'^+, \vey'^-)$, where $\vey'^+$ is obtained from $\vey^+$ by shifting its $i$'th component: $\vey'^+_i := \vey^+_i - \vey^-_i$; and similarly, $\vey'^-$ obtained from $\vey^-$ by setting its $i$'th component equal to zero. Thus, we maintain $B\vex =  \vey'^+ - \vey'^-$ and $\ver' \neq \ve0$, but the support of $(\vey'^+, \vey'^-)$ is strictly contained in that of $(\vey^+, \vey^-)$, a contradiction. The case where $\vey^+_i < \vey^-_i$ can be treated similarly. 

Hence, assume $\vey^+_i = \vey^-_i$ for any index $i$ such that both $\vey^+_i$ and $\vey^-_i$ are nonzero. If there exists any other index $j$ such that $\vey^+_j$ or $\vey^-_j$ are nonzero, then we may again reduce the support of $(\vey^+, \vey^-)$ while retaining the feasibility and nonzero status of $\ver$ by setting both $\vey^+_i$ and $\vey^-_i$ equal to zero. Thus, if there exists any index $i$ with both $\vey^+_i$ and $\vey^-_i$ nonzero, we have $\vey^+ - \vey^- = \ve0$, implying $\vex = \ve0$. Therefore, $\ver$ must be positive multiple of a vector in $T$. 

Thus we may assume that for each index $i \leq m_B$, at most one of $\vey^+_i, \vey^-_i$ is nonzero. Then $\vex$ is nonzero with  $\vey_i^+ = \max\{(B\vex)_i,0\}$ and $\vey_i^- = \max\{ -(B\vex)_i,0\}$ for $i \leq m_B$. Suppose that $\vex$ is not a circuit direction of $P$. Then there exists a circuit $\veg \in \ker(A) \setminus \{\ve0\}$ such that $\supp(B\veg) \subsetneq \supp(B \vex)$. However, this implies that there exists some nonzero $\ver' := (\veg, \vey'^+, \vey'^-) \in S$ such that $(\vey'^+, \vey'^-) \subsetneq (\vey^+, \vey^-)$, contradicting the choice of $\ver$. Therefore, $\vex$ is a circuit direction of $P$ and hence $\ver$ has a positive scalar multiple in $S$.
\eoproof  \end{proof}

Since $(\vey^+, \vey^-)$ consists of nonnegative entries for any $(\vex, \vey^+, \vey^-) \in C_{A,B}$, a normalizing constraint $||\vey^+||_1 + ||\vey^-||_1 = 1$ can be introduced to $C_{A,B}$ by intersecting the cone with a single hyperplane. The extreme rays of $C_{A,B}$ then become vertices of the resulting polytope $P_{A,B}$---a polyhedral model in which the circuits of the original polyhedron $P$ appear as vertices.

\begin{theorem}\label{thm:circuits_as_vertices}
Let $P = \{ \vex \in \R^n \colon A \vex = \veb, B \vex \leq \ved \}$ be a pointed polyhedron. The set of vertices of the polytope
\begin{align*}
P_{A,B} = \{(\vex, \vey^+, \vey^-) \in \R^{n + 2m_B} \colon A\vex = \ve0, \ B\vex = \vey^+ - \vey^-, \ ||\vey^+||_1 + ||\vey^-||_1 = 1,   \  \vey^+, \vey^- \geq \ve0\}.
\end{align*}
is $S_1 \cup T'_1$, where $S_1$ consists of the scaled extreme rays from $S$ of $C_{A,B}$ and $T'_1$ consists of the scaled extreme rays from $T'$. 
\end{theorem}
\begin{proof}
Since $\vey^+$ and $\vey^-$ contain only nonnegative variables, the constraint $||\vey^+||_1 + ||\vey^-||_1 = 1$ corresponds to the hyperplane $\sum_{i=1}^{m_B} \vey_i^+ + \sum_{i=1}^{m_B} \vey_i^- = 1$. Each extreme ray of $C_{A,B}$ intersects this hyperplane exactly once. The convex hull of these intersection points gives the polytope $P_{A,B}$. 
\eoproof  \end{proof}

\Cref{thm:circuits_as_rays,thm:circuits_as_vertices} imply that the set of circuits of $P$ may be enumerated via enumeration over either the extreme rays of $C_{A,B}$ or the vertices of $P_{A,B}$. At most $m_B$ of the directions computed during this approach may belong to the set $T$ and hence will not correspond to circuits of $P$, but these are easily identified. The remaining directions from $S$ will be in one-to-one correspondence with the circuits of $P$. Hence, we obtain \Cref{alg:polyhedral_model} for enumerating $\mathcal{C}(A,B)$ using our polyhedral model.

\begin{algorithm}
\caption{Circuit Enumeration via Polyhedral Model}\label{alg:polyhedral_model}
\begin{algorithmic}[1]
\Procedure{PolyhedralModelCircuits}{$A,B$}\Comment{Computes $\mathcal{C}(A,B)$}
\State $S \gets \emptyset$
\State Use any vertex enumeration algorithm to compute the set $V$ of vertices of $P_{A,B}$.
\For{each $(\vex, \vey^+, \vey^-) \in V$}
\If{$(\vey^+ - \vey^-) \neq \ve0$}
\State $\veg \gets \vex$ scaled to coprime integer components
\State $S \gets S \cup \{\veg\}$
\EndIf
\EndFor
\State \textbf{return} $S$
\EndProcedure
\end{algorithmic}
\end{algorithm}

Although the algorithm requires enumeration over a polytope in $\R^{n + 2m_B}$---whereas $P$ originally belongs to $\R^n$---the advantage of \Cref{alg:polyhedral_model} lies in the fact that any vertex enumeration scheme may be used to compute the vertices of $P_{A,B}$. Currently, most vertex enumeration schemes are variations the Double Description method \cite{f-98} and the Avis-Fukuda method of pivoting using reverse search \cite{af-92}. While most current methods for enumerating circuits are also variants of the Double Description method \cite{gk-04}, \Cref{thm:circuits_as_vertices} implies that Avis-Fukuda pivoting methods for vertex enumeration can be used to enumerate circuits as well. In both degenerate and non-degenerate cases, pivoting methods often outperform Double Description methods---especially when parallelization is utilized---and require significantly less memory \cite{aj-15}.

\Cref{alg:polyhedral_model} also has important theoretical implications. For both the circuit enumeration problem and the vertex enumeration problem for polytopes, it is open whether or not there exists a polynomial total-time algorithm \cite{k-08}---one whose running time is polynomial in both the input and the output sizes. However, \Cref{thm:circuits_as_vertices} implies that a polynomial total-time algorithm for vertex enumeration would immediately yield a polynomial total-time scheme for circuit enumeration.

For a polyhedron $P = \{ \vex \in \R^n \colon A \vex = \veb,  \vex \geq \ve0 \}$ in standard form, the results of Theorems~\ref{thm:circuits_as_rays} and \ref{thm:circuits_as_vertices} can be simplified by reducing the number of variables in the polyhedral model. 

\begin{corollary}\label{cor:standard_circuits_as_rays}
Let $P = \{ \vex \in \R^n \colon A \vex = \veb,  \vex \geq \ve0 \}$ be a polyhedron in standard form. The pointed cone
\begin{align*}
C_{A} = \{( \vey^+, \vey^-) \in \R^{2n} \colon A( \vey^+ - \vey^-) = \ve0,  \ \vey^+, \vey^- \geq \ve0\}.
\end{align*}
is generated by the set of extreme rays $S \cup T'$, where:
\begin{enumerate}
\item The set $S := \{ (\vey^+, \vey^-) \colon   \vey_i^+ = \max\{\veg_i,0\}, \  \vey_i^- = \max\{-\veg_i,0\},  \ \veg \in \mathcal{C}(A)\}$ gives the circuits of $P$.
\item The set $T'$ is a subset of $T := \{ (\vey^+, \vey^-) \colon \vey^+_i = \vey^-_i = 1 \text{ for some $i \leq n$}, \  \vey^+_j = \vey^-_j = 0 \text{ for $j \neq i$}\}$ and has size at most $n$.
\end{enumerate}
\end{corollary}
\begin{proof}
The system of inequality constraints $B \vex \geq \ved$ for $P$ consists of $B = I$ and $\ved = \ve0$. (Switching the direction of inequality constraints does not change the set of circuits of a polyhedron.) Hence, $P$ is pointed and we apply \Cref{thm:circuits_as_rays}. Furthermore, we may eliminate the $\vex$ variables using the substitution $\vex := \vey^+ - \vey^-$. It holds that the set of extreme rays of $C_A$ contains $S$ and is a subset of $S \cup T$. In fact, unless $A$ contains a null column, the extreme rays of $C_A$ will be precisely $S \cup T$ since no unit vector will belong to $\ker(A)$. Each null column of $A$ corresponds to one element of $T$ that is not an extreme ray of $C_A$. 
\eoproof  \end{proof}

Again, by intersecting the pointed cone $C_A$ with a normalizing hyperplane, the circuits of $P$ appear as vertices of the resulting polytope $P_A$.

\begin{corollary}\label{cor:standard_circuits_as_vertices}
Let $P = \{ \vex \in \R^n \colon A \vex = \veb,  \vex \geq \ve0 \}$ be a polyhedron in standard form. The set of vertices of the polytope
\begin{align*}
P_A = \{( \vey^+, \vey^-) \in \R^{2n} \colon A( \vey^+ - \vey^-) = \ve0, \ ||\vey^+||_1 + ||\vey^-||_1 = 1,  \ \vey^+, \vey^- \geq \ve0\}
\end{align*}
is $S_1 \cup T'_1$, where $S_1$ consists of the scaled extreme rays from $S$ of $C_A$ and $T'_1$ consists of the scaled extreme rays from $T'$. 
\end{corollary}

We note that modified versions of the polyhedral models $C_A$ and $P_A$ from \Cref{cor:standard_circuits_as_rays,cor:standard_circuits_as_vertices} appear in the literature. For instance, $C_A$ is used in \cite{gk-04} to enumerate the elementary modes of a metabolic network modeled by the cone $\{ \vex \in \R^n \colon A \vex = \ve0, \vex \geq \ve0 \}$. Versions of $P_A$ are used in \cite{gdl-14} and \cite{gdl-15} to compute strictly improving search directions at degenerate vertices when solving a linear program.  \Cref{cor:standard_circuits_as_vertices} implies that these directions, in fact, are circuits of the underlying polyhedron.

\section{Applications of the Polyhedral Model}\label{sec:advantages}

In this section, we describe key advantages of using the polyhedral model from \Cref{sec:model}.  Namely, in addition to complete enumeration, the model may be used to directly enumerate desirable subsets of circuits (\Cref{sec:subset_enumeration}) and optimize over these sets of circuits (\Cref{sec:optimization_over_circuits}).

\subsection{Modeling Subsets of Circuits}\label{sec:subset_enumeration}
Certain faces of the polyhedral model from \Cref{sec:model} enable the direct enumeration of useful subsets of circuits. Specifically, we show how to use the model to represent only those circuits which are strictly feasible at a given point in a polyhedron (\Cref{sec:feasible}) or only those which are sign-compatible with a given direction (\Cref{sec:sign_compatible}).

\subsubsection{Strictly Feasible Circuits}\label{sec:feasible}

In an augmentation algorithm for solving linear programs, as in \cite{dhl-15,gdl-15,how-11}, a strictly feasible improving search direction is required at each iteration. Given a point $\vex_0$ in a polyhedron $P$, a direction $\veu$ is said to be \textit{strictly feasible} at $\vex_0$ if $\vex_0 + \alpha \veu \in P$ for some $\alpha > 0$. The algorithms of \cite{gdl-14} and \cite{gdl-15} solve a pricing problem over a modified version of the polytope of Corollary~\ref{cor:standard_circuits_as_vertices} in order to find such a direction in a standard form polyhedron. Thus, the search directions used by these algorithms are in fact circuits.

In the following theorem, we generalize the methods of \cite{gdl-14,gdl-15} by using a face of the polyhedral model $C_{A,B}$ from Theorem~\ref{thm:circuits_as_rays} to model all strictly feasible directions at a given point in any general polyhedron. This face may then be used either to enumerate all strictly feasible circuits at a given point or to compute a single search direction as required by an augmentation algorithm.

\begin{theorem}\label{thm:strictly_feasible}
Let $P = \{ \vex \in \R^n \colon A \vex = \veb, B \vex \leq \ved \}$ be a pointed polyhedron and let $\vex_0 \in P$ be given. Consider the face $C_{A,B,\vex_0}$ formed by intersecting the cone
\begin{align*}
C_{A,B} = \{(\vex, \vey^+, \vey^-) \in \R^{n + 2m_B} \colon A\vex = \ve0, \ B\vex = \vey^+ - \vey^-, \ \vey^+, \vey^- \geq \ve0\}
\end{align*}
with the hyperplanes $\vey^+_i = 0$ for each $i \leq m_B$ such that $(B\vex_0)_i = \ved_i$. The extreme rays of $C_{A,B,\vex_0}$ (with the exception of those at most $m_B$ rays for which $\vey^+ - \vey^- = \ve0$) give the strictly feasible circuit directions at $\vex_0$ in $P$. Furthermore, each feasible direction $\veu$ at $\vex_0$ in $P$ has a representative $(\veu, \vey^+, \vey^-)$ in $C_{A,B,\vex_0}$. 
\end{theorem}

\begin{proof}
As a face of $C_{A,B}$, the extreme rays of $C_{A,B,\vex_0}$ are a subset of the extreme rays of $C_{A,B}$. Recall that the extreme rays of $C_{A,B}$ can be partitioned into the sets $S$ and $T' \subseteq T$ as defined in \Cref{thm:circuits_as_rays}. We show that an extreme ray $(\veg, \vey^+, \vey^-) \in S$ is an extreme ray of $C_{A,B, \vex_0}$ if and only if $\veg$ is a strictly feasible circuit at $\vex_0$. Hence, these extreme rays of $C_{A,B, \vex_0}$ correspond to the strictly feasible circuits at $\vex_0$ in $P$. The remaining at most $m_B$ extreme rays of $C_{A,B, \vex_0}$ must come from $T$.

First, suppose $(\veg, \vey^+, \vey^-)$ is an extreme ray of $C_{A,B, \vex_0}$ that belongs to the set $S$ of Theorem~\ref{thm:circuits_as_rays}. Then $(B\veg)_i \leq 0$ for each $i$ such that $(B\vex_0)_i = \ved_i$. If the exists no $i \leq m_B$ such that $(B\veg)_i > 0$, we have $\vex_0 + \alpha \veg \in P$ for any $\alpha > 0$ since $\veg \in \ker(A)$. Therefore, assume there exists at least one index $i$ such that $(B\veg)_i > 0$, and set $\alpha := \min\{ \frac{1}{(B \veg)_i}(\ved_i - (B \vex_0)_i) \colon (B\veg)_i>0 \}$. Then $\alpha > 0$ and for each $i \leq m_B$ such that $(B\veg)_i > 0$:
\begin{align*}
(B(\vex_0 + \alpha \veg))_i  = (B \vex_0)_i + \alpha (B \veg)_i \leq \ved_i.
\end{align*}
Hence, $\vex_0 + \alpha \veg \in P$, so $\veg$ is a strictly feasible circuit. All extreme rays of $C_{A,B,\vex_0}$ that do not belong to the set $T$ of \Cref{thm:circuits_as_rays} must correspond to such a circuit. 

Conversely, if $\veg$ is a strictly feasible circuit at $\vex_0$, it must hold that the corresponding extreme ray $(\veg, \vey^+, \vey^-) \in S$ of $C_{A,B}$ satisfies $(B\veg)_i \leq 0$ and subsequently $\vey^+_i = 0$ for each $i \leq m_B$ where $(B \vex_0)_i = \ved_i$. Therefore, $(\veg, \vey^+, \vey^-)$ belongs to $C_{A,B,\vex_0}$ and must be an extreme ray of $C_{A,B,\vex_0}$. 

Finally, we show that any strictly feasible direction $\veu$ at $\vex_0$ has a representative in $C_{A,B, \vex_0}$. Using results from the upcoming \Cref{sec:sign_compatible}, we obtain a short proof. Namely, since $\veu$ belongs to $\ker(A)$, it holds by \Cref{prop:conformal_sum} that $\veu$ can be expressed as a sum of sign-compatible circuits: $\veu = \sum_{i=1}^t \lambda_i \veg_i$ with $\lambda \geq \ve0$. \Cref{cor:sign-compatible_feasible} implies that each $\veg_i$ used in this sum is a strictly feasible circuit at $\vex_0$ and hence corresponds to an extreme ray $\ver_i \in S$ of $C_{A,B,\vex_0}$. Thus, the conic combination $\sum_{i=1}^t \lambda_i \ver_i$ of these extreme rays yields a representative $(\veu, \vey^+, \vey^-)$ of $\veu$ in $C_{A,B,\vex_0}$ where $\vey_i^+ = \max\{(B\veu)_i,0\}$ and  $\vey_i^- = \max\{-(B\veu)_i,0\}$.
\eoproof  \end{proof}

\Cref{thm:strictly_feasible} implies that the augmentation algorithms of \cite{gdl-14,gdl-15} may be generalized to a polyhedron $P$ of any form since the cone $C_{A,B,\vex_0}$ models all feasible directions at a point $\vex_0 \in P$. Furthermore, the set of all strictly feasible circuits at $\vex_0$ may be directly enumerated via an extreme ray enumeration scheme on $C_{A,B,\vex_0}$ or a vertex enumeration scheme on the corresponding polytope $P_{A,B, \vex_0}$ formed by intersecting $C_{A,B,\vex_0}$ with the hyperplane $||\vey^+||_1 + ||\vey^-||_1 = 1$. 

\subsubsection{Sign-compatible Circuits}\label{sec:sign_compatible}The model from \Cref{thm:circuits_as_rays} can also be used to represent only those circuits which are sign-compatible with a given target direction. Two vectors $\vex, \vey \in \R^n$ are said to be \textit{sign-compatible} if they belong to the same orthant of $\R^n$---that is, if $\vex_i \cdot \vey_i \geq 0$ for $i=1,...,n$. As in \cite{bdf-16}, we then say that $\vex$ and $\vey$ are \textit{sign compatible with respect to a matrix $B$} if the corresponding vectors $B \vex$ and $B \vey$ are sign-compatible.  

An important property of the set of circuits of a polyhedron is the so-called \textit{conformal sum} property \cite{o-10}, due to Graver. For a polyhedron $P = \{ \vex \in \R^n \colon A \vex = \veb,  \vex \geq \ve0 \}$ in standard form, the property states that any direction $\veu \in \ker(A)$ can be expressed as a conformal sum of circuits: $\veu = \sum_{i=1}^t \lambda_i \veg_i$ where $\lambda \geq \ve0$ and the circuits $\veg_1,...,\veg_t \in \mathcal{C}(A)$ are sign-compatible with each other and with $\veu$. In fact, the set of circuits is the unique inclusion-minimal set of directions that has this property \cite{g-75}, which is the basis of several algorithms used to enumerate the set of circuits of a standard form polyhedron \cite{h-03}. In the following \Cref{prop:conformal_sum}, we state this property generalized to a polyhedron in any form.

\begin{proposition}[\cite{g-75}]\label{prop:conformal_sum}
Let $P = \{ \vex \in \R^n \colon A \vex = \veb, B \vex \leq \ved \}$ be a polyhedron with circuits $\mathcal{C}(A,B)$. Any vector $\veu \in \ker(A)$ is a sum of sign-compatible circuits. That is, $\veu = \sum_{i=1}^t \lambda_i \veg_i$ where $\veg_i \in \mathcal{C}(A,B)$, $\lambda_i \geq 0$, and $\veg_i$ is sign-compatible with $\veu$ with respect to $B$ for each $i \leq t$.
\end{proposition}
\noindent In addition, it can be shown that there exists a sign-compatible sum $\veu = \sum_{i=1}^t \lambda_i \veg_i$ in which $\supp(B\veg_i) \centernot\subseteq \bigcup_{j > i}\supp(B\veg_{j})$ for each $i$, implying linear independence of the the $B\veg_i$'s and hence $t \leq n - \rank(A)$ \cite{o-10}.

We show that Proposition~\ref{prop:conformal_sum} is a direct consequence of Theorem~\ref{thm:circuits_as_rays}. Namely, a face of the cone $C_{A,B}$ can be used to model the complete subset of circuits which are sign-compatible with respect to a given direction $\veu \in \ker(A)$.

\begin{theorem}\label{thm:sign-compatible_rays}
Let $P = \{ \vex \in \R^n \colon A \vex = \veb, B \vex \leq \ved \}$ be a pointed polyhedron and let $\veu \in \ker(A)$ be given. Consider cone $C_{A,B,\veu}$ formed by intersecting 
\begin{align*}
C_{A,B} = \{(\vex, \vey^+, \vey^-) \in \R^{n + 2m_B} \colon A\vex = \ve0, \ B\vex = \vey^+ - \vey^-, \ \vey^+, \vey^- \geq \ve0\}
\end{align*}
with the following hyperplanes for each $i \leq m_B$:
\begin{enumerate}
\item $\vey^-_i = 0$ if $(B\veu)_i \geq 0$
\item  $\vey^+_i = 0$ if $(B\veu)_i \leq 0$.
\end{enumerate}
Then $C_{A,B,\veu}$ is a face of $C_{A,B}$ whose extreme rays correspond to the circuits of $P$ which are sign-compatible with $\veu$ with respect to $B$.
\end{theorem}
\begin{proof}
The set $R$ of extreme rays of $C_{A,B, \veu}$ is a subset of the extreme rays of $C_{A,B}$. Let $S$ and $T$ be the sets of extreme rays of $C_{A,B}$ given by Theorem~\ref{thm:circuits_as_rays}. Clearly no ray of $T$ can belong to $R$ since either $\vey^+_i = 0$ or $\vey^-_i = 0$ for each $i \leq m_B$ in each vector of $C_{A,B, \veu}$. Any $(\veg, \vey^+, \vey^-) \in S$ belongs to $C_{A,B, \veu}$ if and only if $B\veg$ is sign-compatible with $B\veu$. Finally, as a face of $C_{A,B}$, a ray of $C_{A,B, \veu}$ belongs to $R$ if and only if it is also an extreme ray of $C_{A,B, \veu}$. 
\eoproof  \end{proof}

Hence, for any $\veu \in \ker(A)$, \Cref{thm:sign-compatible_rays} implies that the subset of $\mathcal{C}(A,B)$ that is sign-compatible with $\veu$ with respect to $B$ can be directly enumerated via extreme ray enumeration over $C_{A,B,\veu}$ or via vertex enumeration over the corresponding polytope $P_{A,B, \veu}$. This approach is a vast improvement over a naive approach for enumerating all such circuits: compute every $\veg \in \mathcal{C}(A,B)$ satisfying $\supp(B\veg) \subseteq \supp(B \veu)$ and then remove those which are not sign-compatible. Another possible way to enumerate sign-compatible circuits would be to use a variation of Pottier's algorithm to find a minimal generating set of $\ker(A)$ intersected with the appropriate cone of $\R^n$ (see Lemma 2 in \cite{bdf-16}). However, this approach also requires the computation of a generating set initially larger than the desired subset of circuits \cite{h-03}.  

Additionally, while previous proofs of the conformal sum property rely on induction or an algorithmic approach (see for instance the proof of \Cref{thm:algorithm}), Theorem~\ref{thm:sign-compatible_rays} provides an immediate proof of Proposition~\ref{prop:conformal_sum}. 
\begin{proof}[of \Cref{prop:conformal_sum}]
Any $\veu \in \ker(A)$ has a corresponding vector $\ver := (\veu, \vey^+, \vey^-) \in C_{A,B}$ such that $\vey_i^+ = \max\{(B\veu)_i,0\}$ and $\vey_i^- = \max\{-(B\veu)_i,0\}\}$ for each $i \leq m_B$. It follows that $\ver \in C_{A, B, \veu}$, implying that $\ver$ can be expressed as a conic combination of extreme rays of $C_{A,B,\veu}$: 
\begin{align*}
(\veu, \vey^+, \vey^-) = \sum_{i=1}^t \lambda_i (\veg_i, (\vey^+)^i, (\vey^-)^i), \ \lambda \geq \ve0.
\end{align*}
Thus, $\veu = \sum_{i=1}^t \lambda_i \veg_i$, where each $\veg_i$ is sign-compatible with $\veu$ with respect to $B$ by Theorem~\ref{thm:sign-compatible_rays}. In addition, the bound $t \leq n - \rank(A)$ follows from Carath\'{e}odory's theorem. 
\eoproof  \end{proof}
\noindent Note that this proof is non-constructive, i.e., it does not provide the circuits which are used in the sign-compatible sum. We will discuss methods for using the polyhedral model to actually construct such a sum in \Cref{sec:construct_sign_compatible_sums}.

We can also use $C_{A,B,\veu}$ to immediately see that sign-compatible circuit directions are strictly feasible. Specifically, let $\vex_1, \vex_2$ be distinct points in a polyhedron and set $\veu := \vex_2 - \vex_1$. It follows from Theorem~\ref{thm:strictly_feasible} that all directions modeled by $C_{A,B,\veu}$ are strictly feasible at $\vex_1$. Hence, in this situation, sign-compatible circuits can be interpreted as a specific subset of strictly feasible directions. We show how to use these circuits to construct walks between $\vex_1$ and $\vex_2$ in \Cref{sec:construct_sign_compatible_sums}.

\begin{corollary}\label{cor:sign-compatible_feasible}
Let $P = \{ \vex \in \R^n \colon A \vex = \veb, B \vex \leq \ved \}$ be a pointed polyhedron and let $\vex_1, \vex_2$ be distinct points in $P$. Set $\veu := \vex_2 - \vex_1$ and consider $C_{A,B,\veu}$ as defined in Theorem~\ref{thm:sign-compatible_rays}. The sign-compatible circuits corresponding to the extreme rays of $C_{A,B,\veu}$ provide strictly feasible circuit directions at $\vex_1$ in $P$.
\end{corollary}
\begin{proof}
Note that $\veu \in \ker(A) \setminus\{0\}$, so $C_{A,B,\veu}$ is well-defined. Any index $i$ such that $(B \vex_1)_i = \ved_i$ must also satisfy $(B \veu)_i \leq 0$, else we would have $(B \vex_2)_i = B(\vex_1)_i + (B \veu)_i > \ved_i$. Therefore, the hyperplane $\vey_i^+ = 0$ is included in the definition of $C_{A,B,\veu}$ whenever $(B \vex_1)_i = \ved_i$. Hence, $C_{A,B,\veu}$ is a face of $C_{A,B,\vex_1}$ as defined in Theorem~\ref{thm:strictly_feasible}, and all of its extreme rays give strictly feasible circuit directions at $\vex_1$ in $P$. 
\eoproof  \end{proof}

\subsection{Optimization over Circuits}\label{sec:optimization_over_circuits}

Since circuits appear as vertices in the polyhedral model $P_{A,B}$ of~\Cref{thm:circuits_as_vertices}, we can efficiently optimize over the set of normalized circuits of a polyhedron via linear programming. Additionally, by considering certain faces of the model as in \Cref{sec:feasible,sec:sign_compatible}, we may optimize over only the subset of strictly feasible or sign-compatible circuits. We show in \Cref{sec:steepest_descent} how this enables the efficient computation of a steepest-descent circuit in an augmentation scheme for solving general linear programs. In \Cref{sec:construct_sign_compatible_sums}, we show how optimization over sign-compatible circuits can be used in the construction of walks between given points of a polyhedron.

\subsubsection{Steepest-descent Circuit Augmentation}\label{sec:steepest_descent}
In \cite{dhl-15}, De Loera et al. describe an augmentation scheme for solving linear programs in standard form that uses a \textit{steepest-descent} circuit at each step. Namely, for a linear program $LP = \min\{ \vecc^T \vex \colon A \vex = \veb, \ \vex \geq \ve0 \}$  with feasible solution $\vex_0$, the authors define a steepest-descent circuit to be a strictly feasible circuit $\veg \in \mathcal{C}(A)$ at $\vex_0$ that minimizes $\vecc^T \veg / ||\veg||_1$ over all such circuits. They proceed to show that an augmentation scheme using only maximal steepest-descent circuit augmentations will never repeat a circuit direction, implying that the algorithm terminates in at most $|\mathcal{C}(A)|$ steps \cite{dhl-15}. However, the scheme relies on an oracle to provide the required steepest-descent circuits.

Independently, Gauthier et al. \cite{gdl-15} describe an algorithmic framework for solving linear programs in standard form in which augmenting directions are computed by fixing subsets of dual variables and then optimizing over a polyhedral oracle. In the special case where no dual variables are fixed (a generalization of the minimum-mean cycle-cancelling algorithm for bounded linear programs \cite{gdl-14}), the authors compute an augmenting direction by minimizing $\vecc^T (\vey^+ - \vey^-)$ over a face of the polytope of \Cref{cor:standard_circuits_as_vertices} (see Equation (7) in \cite{gdl-15} or Equation (15) in \cite{gdl-14}). Since the vertices of this polytope correspond to normalized circuits of $\mathcal{C}(A)$, the computed augmenting direction is in fact a circuit that minimizes $\vecc^T \veg / || \veg||_1$. Thus, the augmenting direction of \cite{gdl-14,gdl-15}, which is computed by solving a linear program, is equivalent to a steepest-descent circuit proposed in \cite{dhl-15}.  

However, both \cite{dhl-15} and \cite{gdl-15} describe augmentation schemes only for polyhedra in standard form. As seen in Section~\ref{sec:representation}, simply converting a general polyhedron to standard form can introduce exponentially many circuits. Since the bound given in \cite{dhl-15} on the number of steps for the steepest-descent augmentation scheme is $|\mathcal{C}(A)|$, these added circuits may significantly affect the performance of the algorithm. Therefore, we generalize the scheme for a polyhedron in general form.

We first generalize the definition of a steepest-descent circuit in order to take into account the differences between circuits of a standard form polyhedron and circuits of a general polyhedron. 
\begin{definition}\label{def:steepest}
Consider the general linear program $LP = \min\{ \vecc^T \vex \colon A \vex = \veb, \ B\vex \leq \ved \}$ with a feasible solution $\vex_0$. A steepest-descent circuit at $\vex_0$ is a strictly feasible circuit $\veg \in \mathcal{C}(A,B)$ that minimizes $\vecc^T \veg / ||B \veg||_1$ over all such circuits. 
\end{definition}
\noindent Note that this definition uses the norm of the vector $B\veg$ rather than the norm of the circuit $\veg$. 

In the remainder of this section, we show that such a circuit exhibits the same desirable properties as the steepest-descent circuits of \cite{dhl-15}, implying that an augmentation scheme using these generalized directions terminates in at most $|\mathcal{C}(A,B)|$ steps. Furthermore, we may use the polyhedral model for strictly feasible circuits from \Cref{sec:feasible} to compute a steepest-descent circuit in polynomial time.

\begin{theorem}\label{thm:steepest-descent}
Let $P = \{\vex \in \R^n \colon A \vex = \veb, B\vex \leq \ved \}$ be a pointed polyhedron, let $\vex_0 \in P$, and let $\vecc \in \R^n$. A steepest-descent circuit direction at $\vex_0$ with respect to $\vecc$ can be computed in weakly polynomial time. 
\end{theorem}
\begin{proof}
Consider the following linear program:
\begin{align}
\tag{steepest}
\begin{split}\label{equation:program}
\min \ &  \vecc^T \vex \\
\text{s.t.} \  & A\vex = \ve0 \\
& B \vex = \vey^+ - \vey^- \\
& \vey^+_i = 0 \  \ \ \ \   \forall i \colon (B\vex_0)_i = \ved_i\\
 & ||\vey^+||_1 + ||\vey^-||_1 = 1 \\
  & \vey^+, \vey^- \geq \ve0.
\end{split}
\end{align}
A vertex solution $(\veg, \vey^+, \vey^-)$ can be computed in weakly polynomial time. If the optimal objective is zero, then no feasible descent directions exist at $\vex_0$. If the objective is nonzero, \Cref{thm:circuits_as_vertices} implies that $\veg$ is a circuit direction of $P$ with $||B \veg||_1 = 1$. Let $\alpha \in \R^+$ denote the unique positive scalar such that $\alpha \veg \in \mathcal{C}(A, B)$. By \Cref{thm:strictly_feasible}, $\alpha \veg$ is a strictly feasible circuit at $\vex_0$ and any other strictly feasible circuit $\veg'$ at $\vex_0$ has a representative $\veg' / ||B\veg'||_1$ in the domain of LP (\ref{equation:program}). It follows that 
\begin{align*}
\frac{\vecc^T (\alpha \veg)}{|| B (\alpha \veg)||_1} = \frac{\alpha (\vecc^T \veg)}{|\alpha| \cdot ||B \veg||_1} = \frac{\vecc^T \veg }{||B \veg||_1}   = \vecc^T \veg \leq \frac{\vecc^T \veg'}{||B\veg'||_1}.
\end{align*}
Therefore, $\veg$ is a steepest-descent circuit direction at $\vex_0$ with respect to $\vecc$. 
\eoproof  \end{proof}

Thus, LP (\ref{equation:program}) in the proof of \Cref{thm:steepest-descent} serves as an oracle in an actual implementation of the generalized steepest-descent augmentation scheme. To prove the bound on the number of augmentations, we mirror the arguments of De Loera et al. (Lemmas 4-6 and Corollary 7 in \cite{dhl-15}) in our more general setting. The proof of the following lemma, which states that a steepest-descent circuit minimizes $\vecc^T \veu / ||B\veu||_1$ over all feasible directions $\veu$ at $\vex_0$, becomes much simpler using our polyhedral model. The other proofs become a bit more technical since $B\veg$ must be considered instead of $\veg$ for norming and support-minimality.

\begin{lemma}\label{lem:steepest}
For a linear program $LP = \min\{ \vecc^T \vex \colon A \vex = \veb, \ B\vex \leq \ved \}$ over a pointed polyhedron with feasible solution $\vex_0$, a steepest-descent circuit direction $\veg$ satisfies $\vecc^T \veg / ||B \veg||_1 \leq \vecc^T \veu / ||B\veu||_1$ for any feasible direction $\veu$ at $\vex_0$. 
\end{lemma}
\begin{proof}
As in the proof of \Cref{thm:steepest-descent}, solve LP (\ref{equation:program}) to obtain a vertex solution $(\veg, \vey^+, \vey^-)$ which yields a steepest-descent circuit direction $\veg$ at $\vex_0$. By \Cref{thm:strictly_feasible}, any other feasible direction $\veu$ at $\vex_0$ has a representative $\veu / ||B\veu||_1$ in the domain of the LP. It follows that $\vecc^T \veg / ||B \veg||_1 = \vecc^T \veg \leq \vecc^T \veu / ||B\veu||_1$.
\eoproof  \end{proof}

This lemma then implies that the steepness of consecutive steepest-descent augmentations is non-increasing.

\begin{lemma}\label{lem:non_increasing}
Let $LP = \min\{ \vecc^T \vex \colon A \vex = \veb, \ B\vex \leq \ved \}$ be a linear program over a pointed polyhedron with feasible solution $\vex_k$, let $\vex_{k+1} := \vex_k + \alpha_k \veg_k$ be a steepest-descent augmentation relative to $\vex_k$, and let $\alpha_{k+1} \veg_{k+1}$ be a steepest-descent augmentation relative to $\vex_{k+1}$. Then $\vecc^T \veg_k / ||B \veg_k||_1 \leq \vecc^T \veg_{k+1} / ||B \veg_{k+1}||_1$.
\end{lemma}

\begin{proof}
Suppose for the purpose of contradiction that $\vecc^T \veg_k / ||B \veg_k||_1 > \vecc^T \veg_{k+1} / ||B \veg_{k+1}||_1$. Then by expansion, substitution, sorting, and properties of the 1-norm, we obtain:
\begin{align*}
\vecc^T (\alpha_k \veg_k + \alpha_{k+1} \veg_{k+1}) &= \alpha_k ||B\veg_k||_1 \frac{\vecc^T \veg_k}{||B\veg_k||_1} + \alpha_{k+1} ||B\veg_{k+1}||_1 \frac{\vecc^T \veg_{k+1}}{||B\veg_{k+1}||_1} \\
&<\alpha_k ||B\veg_k||_1 \frac{\vecc^T \veg_k}{||B\veg_k||_1} + \alpha_{k+1} ||B\veg_{k+1}||_1 \frac{\vecc^T \veg_{k}}{||B\veg_{k}||_1} \\
&= (\alpha_k ||B\veg_k||_1 + \alpha_{k+1} ||B \veg_{k+1}||_1) \frac{\vecc^T \veg_k}{||B\veg_k||_1} \\
& \leq ||\alpha_k B \veg_k + \alpha_{k+1} B \veg_{k+1}||_1 \frac{\vecc^T \veg_k}{||B\veg_k||_1}.
\end{align*}
It follows that
\begin{align*}
\frac{\vecc^T (\alpha_k \veg_k + \alpha_{k+1} \veg_{k+1})}{||\alpha_k B \veg_k + \alpha_{k+1} B \veg_{k+1}||_1} < \frac{\vecc^T \veg_k}{||B\veg_k||_1}.
\end{align*} Since $\alpha_k \veg_k + \alpha_{k+1} \veg_{k+1}$ is itself a feasible direction at $\vex_k$, this inequality contradicts \Cref{lem:steepest}. 
\eoproof  \end{proof}

Next, a change of orthants of $B \veg_k$ implies a change in the steepness of the augmentation step. 

\begin{lemma}\label{lem:sign_change}
Let $LP = \min\{ \vecc^T \vex \colon A \vex = \veb, \ B\vex \leq \ved \}$ be a linear program over a pointed polyhedron with feasible solution $\vex_0$, and let $\alpha_1 \veg_1,...,\alpha_j \veg_j$ correspond to a sequence of steepest-descent augmentations applied beginning at $\vex_0$. If $B \veg_1$ and $B \veg_j$ do not belong to the same orthant of $\R^{m_B}$, then $\vecc^T \veg_1 / ||B\veg_1||_1 < \vecc^T \veg_j / ||B \veg_j||_1$. 
\end{lemma}

\begin{proof}
Suppose for the purpose of contradiction that $\vecc^T \veg_1 / ||B\veg_1||_1 \geq \vecc^T \veg_j / ||B \veg_j||_1$. By \Cref{lem:non_increasing}, it then must hold that 
\begin{align*}
\frac{\vecc^T \veg_1}{||B\veg_1||_1} = \frac{\vecc^T \veg_2}{||B\veg_2||_1}  = \cdots = \frac{\vecc^T \veg_j}{||B\veg_j||_1} .
\end{align*}
Therefore,
\begin{align*}
\vecc^T \left( \sum_{i=1}^j \alpha_i \veg_i \right) = \sum_{i=1}^j \alpha_i ||B \veg_i||_1 \frac{\vecc^T \veg_i}{||B \veg_i||_1} 
&= \left( \sum_{i=1}^j \alpha_i ||B \veg_i||_1 \right) \frac{\vecc^T \veg_1}{||B \veg_1||_1} \\
&<  \left|\left| \sum_{i=1}^j \alpha_i B \veg_i \right|\right|_1 \frac{\vecc^T \veg_1}{||B \veg_1||_1},
\end{align*}
where the last inequality holds because $B \veg_1$ and $B\veg_j$ belong to different orthants of $\R^{m_B}$. Thus
\begin{align*}
\frac{\vecc^T \left( \sum_{i=1}^j \alpha_i \veg_i \right)}{\left|\left| B \left( \sum_{i=1}^j \alpha_i  \veg_i \right) \right|\right|_1} < \frac{\vecc^T \veg_1}{||B \veg_1||_1},
\end{align*}
and since $\sum_{i=1}^j \alpha_i \veg_i$ is itself a feasible direction at $\vex_0$, this again contradicts the choice of $\veg_1$ by \Cref{lem:steepest}. 
\eoproof  \end{proof}

It now follows that no circuit will be repeated as a steepest-descent direction, and the number of augmentations required in the generalized steepest-descent scheme is at most $|\mathcal{C}(A,B)|$.

\begin{corollary}\label{cor:steepest_descent_bound}
Let $LP = \min\{ \vecc^T \vex \colon A \vex = \veb, \ B\vex \leq \ved \}$ be a linear program over a pointed polyhedron with feasible solution $\vex_0$. In a steepest-descent circuit augmentation scheme using maximal steps starting at $\vex_0$, no circuit will be used more than once. It follows that the number of augmentations needed is at most $|\mathcal{C}(A,B)|$.
\end{corollary}
\begin{proof}
Let $\alpha_1 \veg_1,...,\alpha_j \veg_j$ be a sequence of steepest-descent augmentations starting at $\vex_0$. Without loss of generality, it will suffice to show $\veg_1 \neq \veg_j$. Assume the contrary. Then $\vecc^T \veg_1 / ||B \veg_1||_1 = \vecc^T \veg_j / ||B \veg_j||_1$, and by \Cref{lem:sign_change}, each $B\veg_i$ must belong to the same orthant as $B \veg_1$. However, by properties of conformal sums, this implies that $\vex_0 + \alpha_1 \veg_1 + \alpha_j \veg_j$ is a feasible point of the polyhedron, contradicting the maximality of $\alpha_1$. 
\eoproof  \end{proof}

Hence, the generalized steepest-descent augmentation algorithm behaves the same as the standard form scheme of \cite{dhl-15}. Since we can efficiently compute a steepest-descent direction at each iteration using the polyehdral model, it follows that the algorithm runs in strongly polynomial time over a general polyhedron defined by a totally unimodular matrix. (See  Corollary~4 and the following discussion in \cite{dhl-15}.)

\begin{theorem}\label{thm:weakly_polynomial}
Let $LP = \min\{ \vecc^T \vex \colon A \vex = \veb, \ B\vex \leq \ved \}$ be a linear program over a pointed polyhedron where the constraint matrix $\binom{A}{B}$ is totally unimodular. Given a feasible solution $\vex_0$, the steepest-descent augmentation algorithm terminates in strongly polynomial time.
\end{theorem}
\begin{proof}
A consequence of \Cref{lem:sign_change} is that the number of iterations of the steepest-descent augmentation algorithm is bounded by $m_B$ times the number of different values of $-\vecc^T \veg / ||B\veg||_1$ over all circuits $\veg \in \mathcal{C}(A,B)$. If $\binom{A}{B}$ is totally unimodular, it holds that $\veg \in \{0,1,-1\}^n$ and $B\veg \in \{0,1,-1\}^{m_B}$ for each circuit \cite{bv-17}. Hence, since $|\vecc^T \veg| \leq ||\vecc||_1$ and $||B\veg||_1 \leq m_B$, the number of augmentation steps is at most $||\vecc||_1 (m_B)^2$. Using Diophantine approximation \cite{ft-87}, we may replace $\vecc$ with an objective function $\vecc'$ whose entries are sufficiently small to define an equivalent problem whose corresponding bound $||\vecc'||_1 (m_B)^2$ on the number of steps is strongly polynomial. Furthermore, this new objective $\vecc'$ can be found in strongly polynomial time \cite{ft-87}. 

Finally, a steepest-descent circuit can be computed at each iteration by finding a vertex solution to LP (\ref{equation:program}) with objective $\vecc'$. Since $\binom{A}{B}$ is totally unimodular, Tardos' strongly polynomial algorithm for combinatorial linear programs \cite{t-86} yields an optimal solution in strongly polynomial time. If this solution is not a vertex, a reduction process analogous to that in the proof of Theorem 0.2 from \cite{m-91} can be used obtain a vertex optimal solution in strongly polynomially many steps. Therefore, the steepest-descent augmentation algorithm terminates in strongly polynomial time. 
\eoproof  \end{proof}

\subsubsection{Constructing Sign-compatible Sums and Circuit Walks}\label{sec:construct_sign_compatible_sums}

In this section, we show how to use the polyhedral model  from \Cref{sec:sign_compatible} together with the methods from \Cref{sec:steepest_descent} in the construction of sums of sign-compatible circuits. A particularly useful application of these sums is the construction of \textit{sign-compatible circuit walks} between solutions of an optimization problem that take steps in steepest-descent directions.

Given two vertices $\vev_1, \vev_2$ of a polyhedron $P = \{ \vex \in \R^n \colon A \vex = \veb, B \vex \leq \ved \}$, a \textit{circuit walk} $\vev_1 = \vey_0,\vey_1,...,\vey_k = \vev_2$ of length $k$ from $\vev_1$ to $\vev_2$ is a sequence of $k$ steps in which each step direction is a circuit of $P$: $\vey_{i+1} = \vey_i + \alpha_i \veg_i$ where $\veg_i \in \mathcal{C}(A,B)$ and $\alpha_i > 0$. For $i = 0,...,k$, if $\vey_i \in P$, the walk is said to be \textit{feasible}, and if $\vey_i + \alpha \veg_i \notin P$ for any $\alpha > \alpha_i$, the walk is \textit{maximal}. Finally, a circuit walk is said to be \textit{sign-compatible} if each circuit $\veg_i$ used in the walk is sign-compatible with the target direction $\vev_2 - \vev_1$ with respect to $B$. See \cite{bdf-16} for an exposition on different types of certain walks and their implications on the resulting circuit diameter.

An important open question in the study of circuit walks is whether or not the Circuit Diameter Conjecture \cite{bsy-18}, a relaxation of the famous Hirsch Conjecture, is true. That is, if $P$ is a polyhedron with $f$ facets, must there exist a maximal, feasible circuit walk between any pair of vertices with length at most $f - \dim(P)$? Although open in general, the conjecture is in fact true if the maximality requirement of the circuit walk is relaxed~\cite{bdf-16}. Given $\vev_1, \vev_2 \in P$, a sum of sign-compatible circuits $\vev_2 - \vev_1 = \sum_{i=1}^t \lambda_i \veg_i$ with $t \leq n - \rank(A)$ yields a desired feasible, sign-compatible circuit walk $\vev_1 = \vey_0,\vey_1,...,\vey_t = \vev_2$ by setting $\vey_j = \vev_1 + \sum_{i=1}^j \lambda_i \veg_i$ for $j = 0,...,t$. Furthermore, the steps of this walk may be permuted in any order and the walk will remain feasible \cite{bdf-16}. 

Thus, \Cref{prop:conformal_sum} implies that between any pair of vertices $\vev_1, \vev_2$ of a polyhedron $P$, there exists a sign-compatible circuit walk with length at most $n - \rank(A)$. Actually constructing such a walk would yield a short sequence of transitions from $\vev_1$ to $\vev_2$ using only the circuits of $P$. As seen in \cite{bdf-16,bv-17}, circuit walks in polyhedra from combinatorial optimization often have intuitive interpretations in terms of the underlying problem. Sign-compatible circuit walks exhibit additional desirable properties. For instance, in a standard form polyhedron, a sign-compatible circuit walk only ever increases (decreases) each variable if it needs to be increased (decreased) to reach the destination $\vev_2$. In the context of the partition polytopes of \cite{b-13,bv-17}, for example, this means that any item is exchanged at most once when transitioning between given  clusterings of a set. 

By generalizing proofs of \Cref{prop:conformal_sum} given in the literature \cite{g-75} for polyhedra in standard form, we obtain \Cref{alg:construct_sum} for constructing sums of sign-compatible circuits---or equivalently, sign-compatible circuit walks---in any general polyhedron. Informally, given a target direction $\veu \in \ker(A) \setminus \{\ve0\}$, the algorithm selects any circuit $\veg_i$ sign-compatible with $\veu$, takes a longest step with length $\lambda_i$ in the direction of $\veg_i$ such that the remaining difference $\veu - \lambda_i \veg_i$ remains sign-compatible with $\veu$, and then repeats with the direction $\veu - \lambda_i \veg_i$ until the remaining difference is zero. We prove the correctness of \Cref{alg:construct_sum} in  \Cref{thm:algorithm}.

\begin{algorithm}[h]
\caption{Constructing Sign-compatible Sums}\label{alg:construct_sum}
\begin{algorithmic}[1]
\Procedure{SignCompatibleSum}{$A,B,\veu$}\Comment{Assumes $\veu \in \ker(A)$}
\State Initialize: $\vew \gets \veu$, $i \gets 1$, and $D \gets \binom{A}{B}$
\State $Z(\vew) \gets \{j  \colon (D\vew)_j = 0 \}$
\If{$\rank(D_{Z(\vew)}) = n-1$}
\State $\veg_i \gets $ the circuit of $\mathcal{C}(A,B)$ corresponding to circuit direction $\vew$
\State $\lambda_i \gets \frac{||\vew||}{||\veg_i||}$
\State \textbf{return} $\veg_1,...,\veg_i, \lambda_1,...,\lambda_i$
\Else 
\State $\veg_i \gets $ any $\veg \in \mathcal{C}(A,B)$ where $\supp(B \veg) \subsetneq \supp(B \vew)$ and $B\veg$ is sign-compatible with $B \vew$
\State $\lambda_i \gets \min\{(B\vew)_j / (B \veg_i)_j \colon (B\vew)_j(B \veg_i)_j > 0 \}$
\State $\vew \gets \vew - \lambda_i \veg_i$
\State $i \gets i + 1$
\State \textbf{go to} 3
\EndIf
\EndProcedure
\end{algorithmic}
\end{algorithm}

\begin{theorem}\label{thm:algorithm}
For a pointed polyhedron $P = \{ \vex \in \R^n \colon A \vex = \veb, B \vex \leq \ved \}$ and a given $\veu \in \ker(A) \setminus \{\ve0\}$, \Cref{alg:construct_sum} can be used to construct a sign-compatible sum of circuits $\veu = \sum_{i=1}^t \lambda_i \veg_i$ with $t \leq n - \rank(A)$ in polynomial time.
\end{theorem}

\begin{proof}
By \Cref{lem:rank}, a vector $\vew \in \ker(A) \setminus \{\ve0\}$ is a circuit direction of $P$ if and only if it generates the kernel of a row-submatrix of $D = \binom{A}{B}$ with rank $n-1$; equivalently, we must have $\rank(D_{Z(\vew)}) = n-1$ where $Z(\vew)$ is the set of indices $j$ such that $(D \vew)_j = 0$. Thus, if $\rank(D_{Z(\vew)}) = n-1$, we have $\vew = \frac{||\vew||}{||\veg||} \veg$ for some circuit $\veg \in \mathcal{C}(A,B)$. 

On the other hand, if $\vew$ is not a circuit direction of $P$, we seek a circuit $\veg \in \mathcal{C}(A,B)$ sign-compatible with $\vew$ such that $Z(\veg) \supsetneq Z(\vew)$. Since we must have $\rank(D_{Z(\vew)}) \leq n-2$, there exist indices $j,k \notin Z(\vew)$ such that $\rank(D_{Z(\vew) \cup \{j,k\}}) = \rank(D_{Z(\vew)}) + 2 $, which can be found via Gaussian elimination. Furthermore, Gaussian elimination on this system will yield some $\vey \in \ker(D_{Z(\vew) \cup \{k\}}) \setminus \ker(D_{Z(\vew) \cup \{j\}})$. Without loss of generality, we may assume $(B\vew)_j (B\vey)_j > 0$.

By setting $\alpha := \min\{(B \vew)_\ell / (B \vey)_\ell \colon (B\vew)_\ell (B\vey)_\ell > 0\}$ and subtracting $\alpha \vey$ from $\vew$, we obtain a vector $\vez$ that is nonzero since $(B\vez)_k \neq 0$. Additionally, by choice of $\alpha$, $B\vez$ is sign-compatible with $B\vew$ and the support of $B\vez$ is strictly contained in that of $B \vew$. This implies that $\rank(D_{Z(\vez)}) > \rank(D_{Z(\vew)})$. Repeat this process at most $n - \rank(A)$ times until $\rank(D_{Z(\vez)}) = n-1$. Then $\vez$ must be the direction of a circuit $\veg \in \mathcal{C}(A,B)$ that is sign-compatible with $\vew$ and satisfies $Z(\veg) \supsetneq Z(\vew)$, as desired.

Given such a circuit $\veg$, subtract a suitable multiple of $\veg$ from $\vew$ to reduce the support of $B\vew$ while preserving sign-compatibility, i.e., set $\lambda := \min\{(B \vew)_j / (B \veg)_j \colon (B \vew)_j (B\veg)_j > 0\}$ and $\vew := \vew - \lambda \veg$. This increases the rank of $D_{Z(\vew)}$. Hence, we may repeat this process at most $n - \rank(A)$ times until $\vew$ is itself a circuit direction of $P$.
\eoproof  \end{proof}

The only open-ended step of \Cref{alg:construct_sum} is finding a sign-compatible circuit $\veg \in \mathcal{C}(A,B)$ in line 9. Although computing such a circuit in the naive manner of the above proof allows the algorithm to terminate in polynomial time, in applications we can be more deliberate when choosing a circuit. For example, suppose $\vev_1$ is an initial solution and $\vev_2$ is a known optimal solution to the linear program $\min\{ \vecc^T \vex \colon A \vex = \veb, \ B\vex \leq \ved \}$. We may wish to construct a sign-compatible circuit walk from $\vev_1$ to $\vev_2$ in which the objective function decreases as quickly as possible. For instance, when transitioning between the clusterings of a set given by the partition polytopes of \cite{b-13,bv-17}, such a walk would correspond to a sequence of exchanges in which each individual exchange reduces the objective function by a maximum possible value per item moved. We will call such a circuit walk a \textit{$\vecc$-steepest sign-compatible circuit walk}.

\begin{definition}\label{def:f_optimal}
Let $\vecc \in \R^n$. Given two points $\vev_1, \vev_2$ of a polyhedron $P = \{ \vex \in \R^n \colon A \vex = \veb, B \vex \leq \ved \}$, let  $\vev_1 = \vey_0,\vey_1,...,\vey_k = \vev_2$ be a sign-compatible circuit walk from $\vev_1$ to $\vev_2$ where $\vey_{i+1} = \vey_i + \alpha_i \veg_i$ for $i=0,...,k-1$. The walk is a \textit{$\vecc$-steepest sign-compatible circuit walk} if for $i=0,...,k-1$, circuit $\veg_i$ minimizes $\vecc^T \veu / ||B \veu||_1$ over all directions $\veu$ that are sign-compatible with the remaining difference $\vev_2 - \vey_i $ with respect to $B$ and satisfy $\supp(\veu) \subseteq \supp\left( \vev_2 - \vey_i \right)$. 
\end{definition}

Consider such a $\vecc$-steepest sign-compatible circuit walk $\vev_1 = \vey_0, \vey_1,...,\vey_k = \vev_2$ that uses circuits $\veg_0,...,\veg_{k-1}$ as step directions. At each $\vey_i$, the next step direction $\veg_{i}$ is a steepest-descent circuit that is sign-compatible with the remaining difference $\vev_2 - \vey_i$. Furthermore, the steepness $-\vecc^T \veg_i / ||B \veg_i||_1$ of the circuits used in the walk is non-increasing. By using the method from \Cref{sec:steepest_descent} for computing steepest-descent circuits in conjunction with \Cref{alg:construct_sum}, these walks can be constructed in polynomial time.

\begin{corollary}\label{cor:f_optimal_walks}
Let $\vev_1, \vev_2$ be any two points in a pointed polyhedron $P = \{ \vex \in \R^n \colon A \vex = \veb, B \vex \leq \ved \}$, and let $\vecc \in \R^n$. A $\vecc$-steepest sign-compatible circuit walk from $\vev_1$ to $\vev_2$ with at most $n - \rank(A)$ steps can be constructed in polynomial time.
\end{corollary}
\begin{proof}
Use \Cref{alg:construct_sum} to construct a sum of sign-compatible circuits $\vev_2 - \vev_1 = \sum_{i=1}^t \lambda_i \veg_i$ with $t \leq n - \rank(A)$. However, at line 9 during each loop of the algorithm, compute a sign-compatible circuit by finding a vertex solution $(\veg, \vey^+, \vey^-)$ to the linear program  $LP(\vew) =   \min \{ \vecc^T \vex \colon  (\vex, \vey^+, \vey^-) \in P_{A,B,\vew} \} $. Since $P_{A,B, \vew}$ contains a representative $(\veu, \vey^+, \vey^-)$ for each direction $\veu$ with $\supp(\veu) \subseteq \supp(\vew)$ that is sign-compatible with $\vew$, the circuit $\veg$ minimizes $\vecc^T \veu / ||B \veu||_1$ over all such directions. It follows that the resulting circuit walk $\vev_1 = \vey_0, \vey_1,...,\vey_t = \vev_2$ formed by setting $\vey_j = \vev_1 + \sum_{i=1}^j \lambda_i \veg_i$ for $j = 0,...,t$ is a $\vecc$-steepest sign-compatible circuit walk.
\eoproof  \end{proof}


When a polyhedron is defined by a totally unimodular matrix, we can ensure that these $\vecc$-steepest sign-compatible circuit walks  satisfy an additional useful property. A circuit walk $\vev_1 = \vey_0, \vey_1,...,\vey_t = \vev_2$ is said to be \textit{integral} if each $\vey_j$ has integer components. In the context of combinatorial optimization, a circuit walk should be integral in order to have a natural interpretation in terms of the underlying problem. It is shown in \cite{bv-17} that a totally unimodular constraint matrix implies that any maximal circuit walk in the associated polyhedron is integral. However, a sign-compatible circuit walk constructed via \Cref{alg:construct_sum} need not be maximal. In fact, for some polyhedra there exist pairs of vertices which are not joined by any maximal sign-compatible circuit walk \cite{bdf-16}. Nevertheless, a totally unimodular constraint matrix implies that any sign-compatible circuit walk constructed by \Cref{alg:construct_sum} is indeed integral.

\begin{corollary}\label{cor:tu_walks}
Let $P = \{ \vex \in \R^n \colon A \vex = \veb, B \vex \leq \ved \}$ be an integral pointed polyhedron where $\binom{A}{B}$ is totally unimodular. Given any integral points $\vev_1, \vev_2$ in $P$, an integral, sign-compatible circuit walk from $\vev_1$ to $\vev_2$ with at most $n - \rank(A)$ steps can be constructed via \Cref{alg:construct_sum} in polynomial time.
\end{corollary}
\begin{proof}
Let $\veu := \vev_2 - \vev_1$, so that $\veu$ is integral. Use \Cref{alg:construct_sum} to construct a sign-compatible sum $\veu = \sum_{i=1}^t \lambda_i \veg_i$ of at most $n - \rank(A)$ circuits of $P$. Since $\binom{A}{B}$ is totally unimodular, it holds that $B \veg_i \in \{0,1,-1\}^{m_B}$ for $i=1,...,t$ \cite{bv-17}. It follows inductively that each $\lambda_i$ from line 10 of \Cref{alg:construct_sum} is an integer. Setting $\vey_j = \vev_1 + \sum_{i=1}^j \lambda_i \veg_i$ for $j = 0,...,t$ thus yields an integral, sign-compatible circuit walk from $\vev_1$ to $\vev_2$. 
\eoproof  \end{proof}

\section*{Acknowledgments}
We would like to thank Tamon Stephen for the helpful discussions. We also gratefully acknowledge support through the Collaboration Grant for Mathematicians \textit{Polyhedral Theory in Data Analytics} of the Simons Foundation.

\bibliography{literature}
\bibliographystyle{plain}

\end{document}